\documentclass[a4paper;11pt]{amsart}
\usepackage{mathrsfs}
 \usepackage[arrow,matrix]{xy}
\usepackage{amsmath,amssymb,amscd,bbm,amsthm,mathrsfs}
\newtheorem{thm}{Theorem}[section]
\newtheorem{lem}{Lemma}[section]

\newtheorem{prop}{Proposition}[section]
\newtheorem{rem}{Remark}[section]

\theoremstyle{definition}

\begin{document}
\numberwithin{equation}{section}

 \title[  The Neumann problem  for the   $k$-Cauchy-Fueter complex  and the $L^2$ estimate]{
  The Neumann problem  for the   $k$-Cauchy-Fueter complex  over  $k$-pseudoconvex domains in $\mathbb{R}^4$ and the $L^2$ estimate}
\author {Wei Wang}
\begin{abstract}The $k$-Cauchy-Fueter operator  and complex
are quaternionic counterparts of the Cauchy-Riemann operator and the  Dolbeault complex   in the theory of several complex variables, respectively. To  develop the function theory of several
quaternionic variables, we need to solve the non-homogeneous $k$-Cauchy-Fueter equation  over  a domain
  under the   compatibility condition, which naturally leads to a  Neumann problem.
The method of solving the $\overline{\partial}$-Neumann problem in the theory of several complex variables is applied to  this  Neumann problem. We introduce notions of $k$-plurisubharmonic functions and     $k$-pseudoconvex domains, establish the     $L^2$ estimate   and solve the  Neumann problem    over  $k$-pseudoconvex domains in $\mathbb{R}^4$. Namely, we get a vanishing theorem for the first cohomology group   of the $k$-Cauchy-Fueter complex over  such domains.
\end{abstract}

\thanks{Department of Mathematics, Zhejiang University, Zhejiang 310027, P. R.  China,
Email: wwang@zju.edu.cn.}

\thanks{Supported by National Nature Science
Foundation
  in China (No. 11571305)}
 \maketitle

\section{Introduction}The $k$-Cauchy-Fueter operators, $k=0,1,\ldots$,
are quaternionic counterparts of the Cauchy-Riemann operator   in   complex analysis.
The $k$-Cauchy-Fueter complexes  over multidimensional quaternionic space, which
  play  the role  of Dolbeault complex   in the theory of several complex variables,  are now explicitly known \cite{Wa10} (cf. also  \cite{Ba}  \cite{bS} and in particular \cite{bures} \cite{CSS}
  \cite{CSSS} for $k=1$).
It is quite interesting to  develop the function theory of several
quaternionic variables by analyzing these complexes, as it has been done for the Dolbeault complex. A well known
theorem in the theory of several complex variables states that  the first Dolbeault cohomology of a  domain vanishes if and only if it is pseudoconvex. Many remarkable
results about holomorphic functions can be deduced by considering the non-homogeneous $\overline{\partial}$-equation, which leads to
the study of the  $\overline{\partial}$-Neumann problem (cf. e.g.  \cite{CNS} \cite{CS} \cite{FK} \cite{FS} \cite{K} \cite{K2}). Even on one dimensional quaternionic space, i.e.
  $\mathbb{R}^4$, the   $k$-Cauchy-Fueter complexes
\begin{equation}\begin{split}
0\rightarrow C^\infty \left(\Omega, \odot^{k }\mathbb{C}^{2 }  \right) &
\xrightarrow{{\mathscr D}_0^{(k)}}C^\infty \left(\Omega, \odot^{k-1
}\mathbb{C}^{2 } \otimes \mathbb{C}^{2 } \right)
\xrightarrow{{\mathscr D}_1^{(k)}} C^\infty \left(\Omega,\odot^{k-2
}\mathbb{C}^{2 } \otimes \Lambda^2 \mathbb{C}^{2 } \right)  \rightarrow
0,\end{split}\label{eq:quaternionic-complex-diff}
\end{equation}$k=2,3,\ldots$, are nontrivial,
  where $\Omega$ is a domain in $\mathbb{R}^4$,  $\odot^{p}\mathbb{C}^{2 } $ is the $p$-th symmetric power of $ \mathbb{C}^{2 } $ and $\Lambda^{2} \mathbb{C}^{2 } $ is  the exterior  power  of $ \mathbb{C}^{2 } $.  The first operator  ${\mathscr D}_0^{(k)}$ is  called the
{\it $k$-Cauchy-Fueter operator}.
The  non-homogeneous $k$-Cauchy-Fueter equation
\begin{equation}\label{eq:k-CF-1-eq}
 \mathscr   D_0^{(k) } u=f
\end{equation}over a   domain $\Omega $ is overdetermined, and only can be solved under the {\it   compatibility condition}
\begin{equation}\label{eq:compatibility}
 \mathscr     D_1^{(k) }f=0.
\end{equation}  It is known \cite{CMW} that (\ref{eq:k-CF-1-eq})-(\ref{eq:compatibility}) is solvable over a bounded domain $ {\Omega}$ in $\mathbb{R}^4$ with smooth boundary when $f$ is orthogonal to the
{\it first cohomology group} of the $k$-Cauchy-Fueter complex
\begin{equation*}
    H^1_{ (k) }(\overline{\Omega}):= \ker \mathscr D_1^{(k) }/
  {\rm Im}\,   \mathscr D_0^{(k)}
\end{equation*}which is  finite dimensional,
where $\overline{\Omega}$ is the closure of   $ {\Omega}$. We also know the solvability of (\ref{eq:k-CF-1-eq})-(\ref{eq:compatibility}) over $\mathbb{R}^{4n}$ \cite{Wa10} \cite{Wa-weighted}, from which we can drive Hartogs' phenomenon for {\it $k$-regular functions}, i.e. functions annihilated by the $k$-Cauchy-Fueter operator.  The purpose of this paper is to   solve   (\ref{eq:k-CF-1-eq})-(\ref{eq:compatibility}) under  certain geometric conditions for the domain $\Omega$, i.e. to obtain a  vanishing theorem for the first cohomology group  of the $k$-Cauchy-Fueter complex. See also \cite{Wa-mfd} for a vanishing theorem and Weitzenb\"ock  formula for the $k$-Cauchy-Fueter complex   over   curved  compact quaternionic K\"ahler manifolds   with negative scalar curvature.

In the quaternionic case, we have a family of operators  acting on  $ \odot^{k} \mathbb{C}^2 $-valued functions, because the group of unit quaternions,  SU$(2)$,  has  a family of irreducible representations $ \odot^{k} \mathbb{C}^2 $, $k=0,1,\ldots$,   while the group of unit complex numbers, $S^1$,  has only one irreducible representation space $\mathbb{C}$.
 The $k$-Cauchy-Fueter operators over  $\mathbb{R}^{4 }$ also have the origin in physics: they are
the elliptic version of {\it spin $  k/ 2$  massless field} operators (cf. e.g. \cite{CMW} \cite{EPW} \cite{PR1} \cite{PR2}) over the
Minkowski space.
    $\mathscr D_0^{(1)}\phi=0$ corresponds to the Dirac-Weyl equation   whose solutions correspond to neutrinos;
  $\mathscr D_0^{(2)}\phi=0$ corresponds to  the  Maxwell equation   whose  solutions correspond to
photons;
  $\mathscr D_0^{(3)}\phi=0$ corresponds to   the Rarita-Schwinger
equation;
   $\mathscr D_0^{(4)}\phi=0$ corresponds to   linearized Einstein's equation
 whose  solutions correspond to  weak gravitational fields; etc.

The main difference between the $k$-Cauchy-Fueter complexes and the Dolbeault complex   is that there exist symmetric forms except for   exterior forms. Analysis of  exterior forms is classical, while   analysis of symmetric forms is relatively new. We can  handle such forms by using two-component  notation. Such notation is used by physicists as  two-spinor  notation for the   massless field  operators (cf. e.g. \cite{PR1} \cite{PR2} and references therein). It also appears in the study  of quaternionic manifolds (cf. e.g. \cite{Wa-mfd} and references therein).
We will use complex vector fields in two-component  notation:
  \begin{equation}\label{eq:k-CF} (Z_{AA'}):=\left(
                             \begin{array}{rr}
                               Z_{00'  } &  {Z}_{01' } \\
                               Z_{10'  } &  {Z}_{11' } \\
                                                            \end{array}
                           \right):=\left(
                                      \begin{array}{cc}
                                     \partial_{x_1} +\textbf{i}\partial_{x_2}  & - \partial_{x_3} -\textbf{i}\partial_{x_4}  \\
                                      \partial_{x_3}-\textbf{i}\partial_{x_4}  &\quad\partial_{x_1}-\textbf{i}\partial_{x_2}  \\
                                                                               \end{array}
                                    \right),
\end{equation} where $A=0,1 $, $A'=0',1'$, which  are motivated by the embedding of the  quaternion  algebra   into the algebra of complex $2\times 2$-matrices:
\begin{equation*}
 x_{1}+\mathbf{i}x_{2}+\mathbf{j} x_{3}+\mathbf{k}x_{4}   \longmapsto
 \left(\begin{array}{rr} x_{1} +\textbf{i}x_{2}  & - x_{3} -\textbf{i}x_{4}  \\
                                       x_{3}-\textbf{i}x_{4}  &\hskip 3mm x_{1}-\textbf{i}x_{2}
 \end{array}\right).
 \end{equation*}

 In this paper, the method of  solving the $\overline{\partial}$-Neumann problem is extended to solve the corresponding  Neumann problem  for the   $k$-Cauchy-Fueter complexes.   For simplicity, we will drop the superscript $(k)$.
Given a nonnegative measurable function $\varphi$, called a {\it weight  function}, consider the Hilbert  space $L_\varphi^2( \Omega, \mathbb{C}   )$   with the  weighted inner product
\begin{equation}\label{eq:inner-product}
   (a,b)_\varphi:=\int_{ \Omega} a\overline{b}e^{- \varphi}dV,
\end{equation}
where $dV$ is the Lebegues'   measure on $\mathbb{R}^{4 }$.   It induces naturally a weighted $L^2$ inner product  on   $ L_\varphi^2(\Omega,\odot^{p
}\mathbb{C}^{2 } \otimes\Lambda^q \mathbb{C}^{2 })$.
We need to consider ${\mathscr D}_0$ and ${\mathscr D}_1$ as densely defined operators between Hilbert spaces
\begin{equation}\label{eq:quaternionic-complex-diff-L2}
   L_\varphi^2(\Omega,  \odot^{k
}\mathbb{C}^{2 } )\xrightarrow{{\mathscr D}_0}L_\varphi^2(\Omega,\odot^{k-1
}\mathbb{C}^{2 } \otimes \mathbb{C}^{2 } )\xrightarrow{{\mathscr D}_1}L_\varphi^2(\Omega,\odot^{k-2
}\mathbb{C}^{2 } \otimes\Lambda^2 \mathbb{C}^{2 }),
\end{equation}  given by
\begin{equation}\label{eq:k-CF-}\begin{split}
  ({\mathscr D}_0 u)_{A_2'\ldots A_{k  }'A }:  & =\sum_{A '=0',1'}Z^{A '}_Au_{ A 'A_2'\ldots A_{k  }'  } ,\qquad\qquad {\rm for}\quad u\in C^1(\Omega,  \odot^k \mathbb{C}^{2 }   ), \\
    ({\mathscr D}_1 f)_{ A_3'\ldots A_{k  }'AB}:&= \sum_{A'=0',1'}  Z^{A'}_{[ A}f_{B]A' A_3'\ldots A_{k  }' } ,\qquad\quad {\rm for}\quad f\in C^1(\Omega,  \odot^{k-1
}\mathbb{C}^{2 } \otimes \mathbb{C}^{2 }     ),
\end{split}\end{equation}  where  $ A,B =0,1 $, $ A_2',\ldots, A_{k  }'=0',1'$. Here and in the sequel, $Z_A^{A '}$ and $Z_{A '}^A$  are obtained by raising indices (cf. (\ref{eq:Z_A^A'})), and we   use the notation
$
    f_{ A  A_2 '\ldots A_k'  }:=f_{ A_2 '\ldots A_k' A   }
$
 for convenience. $
H_{[A  B ]  }:=\frac 12 ({H_{ A  B   }}-{H_{B  A  }})
 $ is the {\it antisymmetrisation}.  ${\mathscr D}_1\circ {\mathscr D}_0 =0$ can be checked directly (cf. (\ref{eq:exact})).

For a $C^2$ real function $\varphi$, define
\begin{equation}\label{eq:Levi}
   \mathscr L_k(\varphi;\xi)(x): =-k \sum_{ A,B,A_1', \ldots,A_k'}    Z_{B}^{A_1'}Z_{ (A_1 ' }^{ A} \varphi(x)\cdot \xi_{ A_2'\ldots  A_k') A}
  \overline{ \xi_{ A_2'\ldots A_k' B } }
\end{equation}for any $(\xi_{ A_2 '\ldots A_k' A })\in\odot^{k-1}\mathbb{C}^{2 }\otimes \mathbb{C}^{2 }$, where $(A_1'\ldots A_k')$ is symmetrisation    of indices
(cf. (\ref{eq:sym-0})).
A $C^2$ real function $\varphi$ on $\Omega$ is called  {\it (strictly)  $k$-plurisubharmonic} if   there exists a constant $c\geq 0$ ($c>0$) such that
\begin{equation*}  \mathscr L_k(\varphi;\xi)(x)
\geq    c  |\xi  |^2,
\end{equation*}for any $x\in \Omega$ and $\xi \in\odot^{k-1}\mathbb{C}^{2 }\otimes \mathbb{C}^{2 }$.
A domain in $\mathbb{R}^4$ is called {\it (strictly) $k$-pseudoconvex}
if   there exists a defining function $r$ and a constant  $c\geq 0$ ($c>0$) such that
\begin{equation}\label{eq:pseudoconvex} \mathscr L_k(r;\xi)(x)
\geq    c  |\xi  |^2,\qquad x\in \partial\Omega,
\end{equation}  for
any $(\xi_{ A_2 '\ldots A_k' A})\in\odot^{k-1}\mathbb{C}^{2 }\otimes \mathbb{C}^{2 }$ satisfying
\begin{equation}\label{eq:pseudoconvex-tangent}
   \sum_{A=0,1 } Z_{(A_1'}^Ar(x)\cdot \xi_{A_2 '\ldots A_k' ) A}= 0,\qquad x\in \partial\Omega,
\end{equation} for any $A_1 '\ldots A_k' =0',1'$.
  It plays the role of the Levi
form   in   several complex variables. We will see that
    $k$-pseudoconvexity  of a domain is independent of the choice of   defining functions (cf. Proposition \ref{prop:k-pseudoconvexity}). The space of vectors  $\xi  $ satisfying (\ref{eq:pseudoconvex-tangent}) is   of dimension $\geq 2k-(k+1)=k-1$, since there are $k+1$ equations in (\ref{eq:pseudoconvex-tangent}).
  Note that   $k$-plurisubharmonic functions and $k$-pseudoconvex domains are abundant
  since $\chi_1(x_1^2+x_2^2 )$, $\chi_2( x_3^2+x_4^2)$ and their sum are (strictly) $k$-plurisubharmonic  for any increasing smooth (strictly) convex functions $\chi_1$ and $\chi_2$ over $[0,\infty)$ (cf. Proposition \ref{prop:abundant}), and a small perturbation of a strictly  $k$-plurisubharmonic function is still strictly  $k$-plurisubharmonic.

 Consider the {\it associated Laplacian operator} $\Box_{\varphi }: L_\varphi^2( \Omega,    \odot^{k-1
}\mathbb{C}^{2 } \otimes \mathbb{C}^{2 })\longrightarrow L_\varphi^2( \Omega,   \odot^{k-1
}\mathbb{C}^{2 } \otimes \mathbb{C}^{2 })$ given by
\begin{equation}\label{eq:Box-varphi}
   \Box_{\varphi }:={\mathscr D}_0{\mathscr D}_0^*  + {\mathscr D}_1^*{\mathscr D}_1,
\end{equation}
where ${\mathscr D}_0^* $ and $ {\mathscr D}_1^*$ are the adjoint operators of densely defined operators ${\mathscr D}_0  $ and $ {\mathscr D}_1 $, respectively, and
\begin{equation*}
   {\rm Dom} (\Box_{\varphi }):= \left\{ f\in L_\varphi^2( \Omega,\odot^{k-1
}
 \mathbb{C}^{2 } \otimes \mathbb{C}^{2 }  ); f\in {\rm Dom} ({\mathscr D}_0^*) \cap  {\rm Dom}  ({\mathscr D}_1), {\mathscr D}_0^* f\in {\rm Dom} ({\mathscr D}_0)  ,   {\mathscr D}_1 f\in {\rm Dom}  ({\mathscr D}_1^*)  \right\}.
 \end{equation*}
 We can show that
 $ C^\infty \left(\overline{\Omega},
 \odot^{k-1
}\mathbb{C}^{2 } \otimes \mathbb{C}^{2 }\right)\cap {\rm Dom} ({\mathscr D}_0^*)  $ is dense in ${\rm Dom} ({\mathscr D}_0^*) \cap  {\rm Dom}  ({\mathscr D}_1)$ by the density Lemma \ref{prop:domain-k}, and $f\in C^1 \left(\overline{\Omega},
  \odot^{k-1
}\mathbb{C}^{2 } \otimes \mathbb{C}^{2 }\right)\cap {\rm Dom} ({\mathscr D}_0^*)$ if and only if
 \begin{equation}\label{eq:boundary-condition}
  \sum_{A=0,1 } Z_{(A_1'}^Ar(x)\cdot f_{A_2 '\ldots A_k' ) A}= 0,\qquad x\in \partial\Omega,
\end{equation}on the boundary $\partial\Omega$, for any $A_1 '\ldots A_k' =0',1'$.
Moreover, $F\in C^1 \left(\overline{\Omega},
\odot^{k-2} \mathbb{C}^{2 }\otimes \Lambda^2 \mathbb{C}^{2 }\right)\cap {\rm Dom} ({\mathscr D}_1^*) $ if and only if $ \sum_{B=0,1 } Z_{ (A_2'}^Br\cdot F_{A_3 '\ldots A_k') B A  }=0$.
So we need to  consider the following Neumann problem
\begin{equation}\label{eq:Neumann}
  \left\{
  \begin{array} {l}\Box_{\varphi } f=g,\hskip 66.5mm {\rm in } \quad \Omega, \\
  \sum_{A=0,1 } Z_{(A_1'}^Ar(x)\cdot f_{A_2 '\ldots A_k' ) A}=0, \qquad\hskip 23.5mm  {\rm on } \quad\partial\Omega,
   \\  \sum_{B=0,1 }\sum_{B'=0',1'} Z_{ (A_2'}^Br\cdot Z^{B'}_{|[A}f_{B] |A_3 '\ldots A_k') B' }=0 ,
  \qquad \hskip 1.5mm {\rm on } \quad\partial\Omega,
  \end{array}
  \right.
\end{equation}  for any $A_1 '\ldots A_k' =0',1'$,  where $(\cdots|\mathscr A | \cdots)$ means symmetrisation of indices   except for that in   $\mathscr{A}$.
The key step   is to establish the following    $L^2$ estimate (cf. e. g. \cite{CS} \cite{Hor-L2} for the   $L^2$ estimate for   the $\overline{\partial}$ operator).

\begin{thm}\label{thm:L2-estimate}For fixed $k\in \{2,3,\ldots\}$,
   let $\Omega$ be a bounded $k$-pseudoconvex domain in $\mathbb{R}^4$  with smooth boundary and let  $\varphi$ be  a smooth strictly $k$-plurisubharmonic function, i.e.   \begin{equation}\label{eq:Levi0}
     \mathscr L_k( \varphi;\xi) (x)\geq c  |\xi |^2,\qquad x\in \Omega,
    \end{equation} for some $c>0$ and any $\xi\in\odot^{k-1
}\mathbb{C}^{2 } \otimes\mathbb{C}^{2 }$. Suppose that
 \begin{equation}\label{eq:C0}
   C_0:=\frac {c- 4\|d\varphi\|_\infty^2}{2k+14 } >0,\end{equation}
  where $\|d\varphi\|_\infty^2=\sum_{j=1}^4 \|\frac {\partial\varphi}{\partial x_j}\|_{L^\infty(\Omega)}^2$.
    Then the $L^2$-estimate
    \begin{equation}\label{eq:L2-k}
C_0 \left \| f \right\|^2_{   \varphi  }  \leq      \left\| {\mathscr D}_0^*f  \right\|^2_{ \varphi } +  \left\|{\mathscr D}_1 f \right\|^2_{ \varphi }
    \end{equation}holds
    for any $f\in {\rm Dom} ( {\mathscr D}_0^*)\cap {\rm Dom} ( {\mathscr D}_1)$.
 \end{thm}
 In particular, if $\varphi$ only satisfies the condition (\ref{eq:Levi0}), then  $ \kappa \varphi $ for   $0<\kappa<\frac c{ 4\|d\varphi\|_\infty^2}$  is a weight function satisfying the assumption (\ref{eq:Levi0})-(\ref{eq:C0}) in the above theorem with suitable constants (cf. Remark \ref{eq:C0-kappa} (1)). The {\it   $k$-Bergman space} with respect to   weight $\varphi $ is then defined as
 \begin{equation*}
    A^2_{(k)}(\Omega, \varphi):=\left\{f\in L_\varphi^2( \Omega,  \odot^{k
}\mathbb{C}^{2 } );{\mathscr D}_0 f=0\right\}.
 \end{equation*} It is infinite dimensional \cite{KW} because $k$-regular polynomials are in this space for bounded $\Omega$.

   \begin{thm} \label{thm:canonical} Let the domain $\Omega$ and the weight function $ \varphi$ satisfy   assumptions in the above theorem. Then
     \item[(1)]{ $\Box_\varphi$ has a bounded, self-adjoint and non-negative inverse $N_\varphi$ such that  }
       \begin{equation*}
          \|N_\varphi f\|_\varphi\leq \frac 1{  {C_0}}\|f\|_\varphi, \qquad {\rm for}  \quad {\rm  any}\quad   f\in L_\varphi^2( \Omega,   \odot^{k-1
}\mathbb{C}^{2 } \otimes \mathbb{C}^{2 } ) .
       \end{equation*}

   \item[(2)]{    $ {\mathscr D}_0^*N_\varphi f $
        is  the canonical solution operator to the nonhomogeneous     $k$-Cauchy-Fueter equation (\ref{eq:k-CF-1-eq})-(\ref{eq:compatibility}), i.e. if $f  $ is ${\mathscr D}_1$-closed, then
         $
            {\mathscr D}_0 {\mathscr D}_0^*N_\varphi f=f
      $
        and $   {\mathscr D}_0^*N_\varphi f $ is orthogonal to $A^2_{(k)}(\Omega,\varphi)$. Moreover,
        \begin{equation}\label{eq:canonical-est}
           \| {\mathscr D}_0^*N_\varphi f\|_\varphi^2+ \| {\mathscr D}_1 N_\varphi f\|_\varphi^2  \leq  \frac 1{  {C_0}}\|  f\|_\varphi^2.
        \end{equation}
        }      \end{thm}

     Since  the $k$-Bergman space $A^2_{(k)}(\Omega, \varphi)$ is a closed Hilbert  subspace, we have
   the orthogonal projection $P: L_\varphi^2( \Omega,  \odot^{k
}\mathbb{C}^{2 } ) \longrightarrow A^2_{(k)}(\Omega, \varphi)$,  the {\it    $k$-Bergman projection}. It follows from the above theorem that
 $
          Pf =f- {\mathscr D}_0^*N_\varphi{\mathscr D}_0f
     $
       for $f\in{\rm Dom} ({\mathscr D}_0 ) $, as in the theory of several complex variables (cf. theorem 4.4.5 in \cite{CS}).

              This framework can be applied to   the   $k$-Cauchy-Fueter complex  over  the higher dimensional space. We restrict to $4$-dimensional case because of the difficulty of obtaining the   $L^2$ estimate over $\mathbb{R}^{4n}$ for $n>1$ (cf. Remark \ref{eq:C0-kappa} (2)).
In Section 2, we give the necessary preliminaries on raising or lowering primed or unprimed indices, symmetrisation and antisymmetrisation of indices, the $k$-Cauchy-Fueter operator, the complex vector field $Z_{A}^{A'}$'s and  their formal adjoint operators, etc.. In Section 3, we derive  the  Neumann  boundary condition, introduce   notions of   $k$-plurisubharmonic functions and $k$-pseudoconvex domains, and show their properties mentioned above.
In Section 4.1, the    $L^2$ estimate
 in Theorem \ref{thm:L2-estimate} is established. In Section 4.2, we deduce the density lemma  from a general result due to H\"ormander,  and derive Theorem  \ref{thm:canonical} from  the    $L^2$ estimate
 in Theorem \ref{thm:L2-estimate}.

 I would like to thank the referee for many valuable suggestions.

\section{Preliminary}

\subsection{Symmetrisation  and antisymmetrisation}
Recall that the
{\it symmetric power} $ \odot^{p
}\mathbb{C}^{2 } $ is a subspace of $ \otimes^{p
}\mathbb{C}^{2 } $, and an element of $\odot^{p} \mathbb{C}^2 $ is given by a $2^p$-tuple  $
(  f_{A_1' \ldots A_p'})\in \otimes^{p} \mathbb{C}^2$ with $A_1' \ldots A_p'   =0',1' $  such that
$  f_{A_1' \ldots A_p'} $ is invariant under   permutations of
subscripts, i.e.
\begin{equation*}
     f_{A_1' \ldots A_p'} =  f_{   A_{\sigma(1)}'\ldots A_{\sigma(p)} '}
 \end{equation*} for any $\sigma\in S_p$, the group of permutations of $p$ letters.
We will use  symmetrisation    of indices
\begin{equation}\label{eq:sym-0}
     f_{\cdots (A_1'\ldots A_p')\cdots}: =\frac 1 {p!}\sum_{\sigma\in S_p}  f_{\cdots  A_{\sigma(1)}'\ldots A_{\sigma(p)}' \cdots}.
     \end{equation}In particular, if $(f_{A_1'\ldots A_p'})\in \otimes^p \mathbb{C}^{2 }$ is symmetric in $A_2'\ldots A_p'$, then we have
 \begin{equation}\label{eq:sym-1}
   f_{(A_1'\ldots A_p')}=\frac 1k \left(f_{A_1'A_2'\ldots A_p'}+\cdots+ f_{A_s'A_1'\ldots \widehat{A_s'}\ldots A_p'}+\cdots+ f_{A_p'A_1'\ldots A_{p-1}' } \right).
 \end{equation}
  An element of $f \in L_\varphi^2(\Omega,   \odot^{k-2
}\mathbb{C}^{2 } \otimes \Lambda^2\mathbb{C}^{2 }  )$ is given by a $2^{k }$-tuple $ (  f_{A_3' \ldots A_k'AB})\in L_\varphi^2(\Omega,   \otimes^{k } \mathbb{C}^2)$   such that
they are invariant under   permutations  of primed indices and $f_{\cdots AB}=-f_{\cdots  B A}$.
 For any $f, g\in L_\varphi^2(\Omega,   \odot^{k-2
}\mathbb{C}^{2 } \otimes  \Lambda^2\mathbb{C}^{2 }  )$, define
 \begin{equation}\label{eq:inn-product2}
  \left \langle f, g\right\rangle_{\varphi }=\sum_{  A,B=0,1} \sum_{  A_3', \ldots,A_k'=0',1'}\left(  f_{ A_3'\ldots A_k'AB } ,  g_{ A_3'\ldots A_k'AB }\right)_{\varphi },
 \end{equation} and $\|f\|_\varphi= \left \langle f,  {f}\right\rangle_{\varphi }^{\frac 12}$.
Similarly,  we define the weighted inner products of $ L_\varphi^2(\Omega,  \odot^{k
}\mathbb{C}^{2 } )$ and $ L_\varphi^2(\Omega,\odot^{k-1
}\mathbb{C}^{2 }\otimes \mathbb{C}^{2 } )$ as subspaces of $ L_\varphi^2(\Omega,  \otimes^{k
}\mathbb{C}^{2 } )$. $|\xi|$ for $\xi\in\odot^{k-1
}\mathbb{C}^{2 } \otimes  \mathbb{C}^{2 } $ is defined in the same way.

We use
 \begin{equation}  (\varepsilon_{A'B'})=\left( \begin{array}{cc} 0&
 1\\-1& 0\end{array}\right) \qquad {\rm and } \qquad \left (\varepsilon^{A'B'}\right) =\left( \begin{array}{cc} 0&
- 1\\1& 0\end{array}\right)  \label{eq:varepsilon}
 \end{equation} to raise or lower primed  indices,  where $(\varepsilon^{A'B'})$ is the inverse of $(\varepsilon_{A'B'})$. For example,
\begin{equation*}
    f_{\ldots\phantom{A'}\ldots}^{\phantom{\ldots}A'}=  \sum_{B'=0',1'}   f_{\ldots{B'}\ldots}\varepsilon^{B'A'},\qquad \sum_{A'=0',1'}   f_{\ldots\phantom{A'}\ldots}^{\phantom{\ldots}A'}\varepsilon_{A'C'}=  f_{\ldots{C'}\ldots}.
\end{equation*}Since
$
 \sum_{B'=0',1'}  \varepsilon_{A'B'}\varepsilon^{B'C'}=\delta_{A'}^{ C'}= \sum_{B'=0',1'} \varepsilon^{C'B'}\varepsilon_{B'A'}
$,  it is the same when an index is raised (or lowered) and then lowered  (or raised).
Similarly we use
 \begin{equation}  (\epsilon_{A B })=\left( \begin{array}{cc} 0&
 1\\-1& 0\end{array}\right),\qquad  (\epsilon^{A B }) =\left( \begin{array}{cc} 0&
- 1\\1& 0\end{array}\right)  \label{eq:epsilon}
 \end{equation}
  to raise or lower unprimed indices. We have \begin{equation}\label{eq:Z_A^A'}
  \left ( Z_{A}^{A'}\right) =\left(
                             \begin{array}{rr}
                               Z_{01'  } &  -{Z}_{00' } \\
                               Z_{11'  } &  -{Z}_{10' } \\
                                                            \end{array} \right),
  \qquad  \qquad
   \left( Z^{A}_{A'}\right) =\left(
                             \begin{array}{rr}
                               Z_{ 10'  } &  -{Z}_{00' } \\
                              Z_{11'  } &-{Z}_{ 01' } \\
                                                            \end{array}
                           \right).
 \end{equation}

The following properties of symmetrisation and antisymmetrisation of indices are frequently used later.
\begin{lem}\label{lem:sym} {\rm (cf. lemma 2.1 in \cite{Wa-weighted})} \item[(1)]{For any   $  g, G\in      \otimes^p \mathbb{C}^{2 }    $, we have
   \begin{equation}\label{eq:bracket-omit}
  \sum_{B_1', \ldots, B_p'=0',1'}     g_{ (B_1' \ldots B_p')} \overline{G_{ (B_1' \ldots B_p')}} =\sum_{ B_1', \ldots, B_p'=0',1'
   }     g_{ (B_1' \ldots B_p')} \overline{ G_{  B_1' \ldots B_p' }} .
\end{equation}}

\item[(2)]{For any $(h_{A
  B
  })\in \Lambda^2 \mathbb{C}^2$ and $(H_{A B })\in \otimes^2 \mathbb{C}^2$,
    \begin{equation}\label{eq:antisym1}\sum_{A ,B=0,1 } h_{ A B} \overline{H_{ A B}}=
    \sum_{A ,B=0,1    } h_{A B} \overline{H_{ [A B ]}}.
   \end{equation}}

\item[(3)]{  For any $(h_{A B }), (H_{A B })\in \otimes^2 \mathbb{C}^2$,
   \begin{equation}\label{eq:antisym}
     \sum_{A ,B =0,1   }  h_{B A } \overline{H_{ A  B  }}=  \sum_{A ,B=0,1    }   h_{A B } \overline{H_{ A  B   }}-2  \sum_{A ,B   } h_{[A B] } \overline{H_{ [ A  B]    }}.
   \end{equation}}
\end{lem}

\begin{lem}\label{lem:raise}
   \begin{equation}\label{eq:raise} \begin{split}
     \sum_{A=0,1} f_{\ldots {\phantom A} \ldots A \ldots}^{\phantom {\ldots  }A} =-f_{\ldots 0\ldots 1\ldots} +f_{\ldots 1\ldots 0\ldots} ,\qquad
   \sum_{A'=0',1'}   f_{\ldots\phantom {A'} \ldots A'\ldots}^{\phantom {\ldots   }A'} =-f_{\ldots 0'\ldots 1'\ldots} +f_{\ldots 1'\ldots 0'\ldots}.
  \end{split}  \end{equation}
\end{lem}
This lemma means that   contraction of indices is just   antisymmetrisation:
 \begin{equation*}
    f_{  {\phantom A}   A  }^{ A}=-2f_{[01]},\qquad  f_{\phantom {A'}   A'  }^{ A'}=-2f_{[0'1']},
 \end{equation*}
  which is very important to establish our $L^2$ estimate and only holds   in dimension $4$ (cf. Remark \ref{eq:C0-kappa}).

\subsection{The  formal adjoint operator and Stokes' formula}
  One advantage of
  raising   indices is that the  formal adjoint operator   of $ Z_{A}^{A'}$  can  be written in a very simple form.
For a fixed weight $\varphi$, we introduce   differential operators
\begin{equation*}
   \delta_{A'}^A a:=  Z^{ A}_{A'} a-Z^{ A}_{A'}\varphi\cdot a
\end{equation*}
for a scalar function $a$.

\begin{prop}\label{prop:formal-adjoint} (1) The formal adjoint operator
 of $Z_{A}^{A' }$ with respect to the weighted inner product (\ref{eq:inner-product}) is
    $ \delta_{A'}^A$£¬ i.e. for any $a,b\in C_0^1(\Omega,\mathbb{C})$ we have
\begin{equation}\label{eq:Z_AA'2}
 \left (Z_{A}^{A'} a,b\right)_\varphi=\left( a,\delta_{A'}^Ab\right)_\varphi.
  \end{equation}

(2) We have    \begin{equation}\label{eq:Z-raise}
      \overline{Z_{A}^{A'}}=-Z^{A}_{A'} .
    \end{equation}
\end{prop}

(2) can be checked directly by using definition.
For a defining function $r$ of the domain $\Omega$ with $|dr|=1$ on the boundary and a complex vector field $Z$, we have Stokes' formula
\begin{equation*}\label{eq:Stokes0}
   \int_\Omega Za\cdot \overline{b} e^{- \varphi}dV+\int_\Omega a\cdot  {Z}\overline{b}\cdot e^{- \varphi} dV - \int_{\Omega} a\overline{b} \cdot Z\varphi \cdot e^{- \varphi} dV= \int_{\partial\Omega} Zr\cdot a  \overline{b} e^{- \varphi}dS
\end{equation*}for any $a,b\in C^1(\overline{\Omega},\mathbb{C})$, where $dS$ is the surface measure of the boundary, i.e.
\begin{equation}\label{eq:Stokes}
   (Z a, b)_\varphi = \left ( a, Z^*_\varphi b\right)_\varphi  +\int_{\partial\Omega} Zr\cdot a  \overline{b} e^{- \varphi}dS,
\end{equation} where $
      Z_\varphi^*=-\overline{Z }+ \overline{Z  }\varphi
$.
In particular, we have
\begin{equation}\label{eq:Stokes-Z}
  \left (Z_{A}^{A'} a, {b}\right)_\varphi =  \left( a, \delta_{A'}^A {b}\right)_\varphi+\mathscr B,\qquad  \mathscr B = -\int_{\partial\Omega}    a\cdot \overline{Z^A_{A'}r\cdot b} e^{- \varphi}dS,
\end{equation}
by using  (\ref{eq:Z-raise}),  and by taking conjugate,
\begin{equation}\label{eq:Stokes-delta}
  \left (\delta_{A'}^{A} a, b\right)_\varphi = \left( a, Z_A^{A'} b\right)_\varphi+\mathscr B',\qquad  \mathscr B' =- \int_{\partial\Omega}     a\cdot \overline{Z^{A'}_Ar\cdot b} e^{- \varphi}dS.
\end{equation}

  We have ${\mathscr D}_1\circ {\mathscr D}_0 =0$   (cf. (2.11) in \cite{CMW}) because ${\mathscr D}_0$ and ${\mathscr D}_1$ as differential operators are both densely defined and closed,  and for any $u\in C^2(\Omega,  \odot^k \mathbb{C}^{2 }   )$,
\begin{equation}\label{eq:exact}\begin{split}
    (\mathscr  D_1\mathscr D_0u)_{ A_3'\ldots A_{k  }' AB}  =\frac 12\sum_{A',C'=0',1'}\left( Z^{A'}_{
    A}Z^{C'}_Bu_{C'A'A_3'\ldots A_{k  }' }-Z^{A'}_{B}Z^{C'}_Au_{C'A'A_3'\ldots A_{k  }'  }\right)=0
\end{split}\end{equation}
by relabeling  indices, $ u_{C'A'A_3'\ldots A_{k  }' }=u_{A'C'A_3'\ldots A_{k  }' }$ and   the commutativity $Z^{A'}_{ B}Z^{C'}_A=Z^{C'}_AZ^{A'}_{B}$, which holds for  scalar differential operators of
constant complex coefficients.

We have   isomorphisms
 \begin{equation}\label{eq:isomorphisms}
 \odot^{k }\mathbb{C}^2 \cong \mathbb{C}^{k+1 } , \qquad  \odot^{k-1 }\mathbb{C}^2\otimes  \mathbb{C}^2  \cong \mathbb{C}^{2k},\qquad \odot^{k-2 }\mathbb{C}^2\otimes \Lambda^2\mathbb{C}^2\cong
 \mathbb{C}^{k-1}
 \end{equation}
by identification $u\in\odot^{k }\mathbb{C}^2  $,  $  f\in \odot^{k-1 }\mathbb{C}^2\otimes  \mathbb{C}^2   $, $ F\in  \odot^{k-2 }\mathbb{C}^2\otimes \Lambda^2\mathbb{C}^2$ with\begin{equation}\label{eq:Dk}
     \left(\begin{array}{c}u_{0' \ldots 0'0'}\\u_{1' \ldots 0'0'}
\\ \vdots\\u_{1' \ldots 1' 0'}\\u_{1' \ldots 1'1'}
\end{array}
\right) ,\qquad   \left(\begin{array}{c}f_{ 0'\ldots  0'0}\\ f_{ 0'\ldots  0'1}
\\ \vdots\\ f_{1'\ldots  1' 0}\\ f_{1'\ldots  1' 1 }\end{array}\right),\qquad  \left(\begin{array}{c}F_{ 0'\ldots  0'01}\\ F_{ 1'\ldots  0'01}
\\ \vdots\\ F_{ 1'\ldots  1'0 1}\end{array}\right) ,
\end{equation}   respectively. Then the  $k$-Cauchy-Fueter complex   becomes
\begin{equation*}\begin{split}
0\rightarrow C^\infty \left(\Omega,  \mathbb{C}^{k+1 }  \right) &
\xrightarrow{{\mathscr D}_0 }C^\infty \left(\Omega,       \mathbb{C}^{2k  }  \right)
\xrightarrow{{\mathscr D}_1 } C^\infty \left(\Omega,    \mathbb{C}^{k-1 } \right)  \rightarrow
0.\end{split}
\end{equation*}See also \cite{CMW} for the matrix form of the Neumann problem (\ref{eq:Neumann}) for general $k$. But we use a different norm of $\odot^{k }\mathbb{C}^2$ there.

\section{The Neumann  boundary condition,  $k$-plurisubharmonicity  and $k$-pseudoconvexity}
 \subsection{The Neumann  boundary condition}

\begin{prop}\label{prop:Neumann-k}
(1) For $f\in  C_0^1(\Omega, \odot^{k-1}
 \mathbb{C}^{2 }\otimes \mathbb{C}^{2 } )$, $F\in  C_0^1(\Omega, \odot^{k-2}
 \mathbb{C}^{2 }\otimes\Lambda^2\mathbb{C}^{2 } )$, we have
   \begin{equation}\label{eq:T*}\begin{split}&
   ({\mathscr D}_0^*f)_{A_1 '\ldots A_k'}=\sum_{A =0,1} \delta_{(A_1 ' }^{ A} f_{ A_2 '\ldots A_k' ) A } , \qquad  ({\mathscr D}_1^*F)_{A_2 '\ldots A_k'A }=\sum_{A =0,1} \delta_{ (A_2  ' }^{ B} F_{A_3 '\ldots A_k') B A   }.
  \end{split} \end{equation}

(2)  $f\in C^1(\overline{\Omega}, \odot^{k-1} \mathbb{C}^{2 } \otimes \mathbb{C}^{2 }
)\cap{\rm Dom}( {\mathscr D}_0^*)$ if and only if
 \begin{equation}\label{eq:Neumann-0'}
   \sum_{A =0,1} Z_{(A_1 ' }^Ar\cdot f_{A_2 '\ldots A_k')A}= 0 \qquad {\rm on } \quad\partial\Omega
 \end{equation}for any $ A_1 ,\ldots ,A_k'=0',1'$; and
   $F\in C^1(\overline{\Omega},\odot^{k-2} \mathbb{C}^{2 }\otimes \Lambda^2 \mathbb{C}^{2 }
)\cap  {\rm Dom} ({\mathscr D}_1^*)$ if and only if
 \begin{equation}\label{eq:Neumann-1'}
  \sum_{B=0,1 } Z_{ (A_2'}^Br\cdot F_{A_3 '\ldots A_k') B A  }=0 ,\qquad{\rm on } \quad\partial\Omega
 \end{equation}for any $ A_2 ,\ldots ,A_k'=0',1'$, $A=0,1$ \end{prop}
\begin{proof} (1) For any $u \in C_0^\infty (\Omega,  \odot^k \mathbb{C}^{2 }   )$,
   \begin{equation*}\begin{split}
    \langle{\mathscr D}_0 u, f\rangle_{\varphi}& =\sum_{A, A_1 ',\ldots, A_k'} ( Z^{A_1 ' }_{ A}  u_{ A_1'A_2 '\ldots A_k' } ,  { f_{ A_2 '\ldots A_k' A}} )_{ \varphi }
         =\sum_{  A_1 ',\ldots, A_k'}\left(  u_{ A_1'A_2 '\ldots A_k' } ,\sum_{A } \delta_{(A_1 ' }^{ A} { f_{A_2 '\ldots A_k') A }}\right)_{ \varphi }=\langle u,{\mathscr D}_0^*f\rangle_{\varphi}
 \end{split}  \end{equation*} by  $\delta_{A  ' }^{ A}$ as the formal adjoint operator of $Z^{A ' }_{ A}$ in (\ref{eq:Z_AA'2}).  Here we need to symmetrise primed indices in $\delta_{ A_1 ' }^{ A} { f_{A_2 '\ldots A_k'  A }} $ by using      Lemma \ref{lem:sym} (1), since  only after symmetrisation it becomes a  $\odot^{k
}\mathbb{C}^{2 }$-valued function.

(2)  If $f\in C^1(\overline{\Omega}, \odot^{k-1} \mathbb{C}^{2 }\otimes  \mathbb{C}^2
)\cap  {\rm Dom} ({\mathscr D}_0^*) $, then for any $u\in C^2(\overline{\Omega}, \odot^k
 \mathbb{C}^{2 })$ we have
 \begin{equation*}\begin{split}
 \left \langle {\mathscr D}_0  u ,   f \right\rangle_\varphi  &= \sum_{A, A_1 ',\ldots, A_k'}\left(Z^{ A_1'}_A u_{A_1 '\ldots A_k'}  , f_{A_2 '\ldots A_k' A}\right)_\varphi =
  \sum_{A_1 ',\ldots, A_k'}\left(u_{A_1 '\ldots A_k'} , \sum_A \delta_{(A_1'}^A f_{A_2 '\ldots A_k' )A}\right)_\varphi + \mathscr  B_0 \end{split}  \end{equation*}
 by applying  Stokes' formula
 (\ref{eq:Stokes-Z}) and using  symmetrisation  by   Lemma \ref{lem:sym} (1), with the boundary term
 \begin{equation*}\label{eq:boundary-term0}\begin{split}
 \mathscr  B_0:&= -\sum_{A, A_1 ',\ldots }\int_{\partial\Omega} u_{A_1 '\ldots A_k'}  \overline{Z_{ A_1'}^Ar\cdot  f_{A_2 '\ldots A_k' A}} e^{- \varphi }dS = -\int_{\partial\Omega} \sum_{  A_1 ',\ldots }u_{A_1 '\ldots A_k'} \cdot \overline{\sum_{A  }Z_{ (A_1'}^Ar\cdot  f_{A_2 '\ldots A_k') A}} e^{- \varphi }dS,
  \end{split}\end{equation*}by using  symmetrisation  by Lemma \ref{lem:sym} (1) again.
  Thus
$\langle\mathscr D_0u,f\rangle_\varphi =
  \langle u,\mathscr D_0^*f\rangle_\varphi
$
 if and only if the boundary term $\mathscr  B_0$
  vanishes for any $\odot^{k
}\mathbb{C}^{2 }$-valued function $u$, i.e.  (\ref{eq:Neumann-0'}) holds.

Now for $h\in C^1(\overline{\Omega}, \odot^{k-1} \mathbb{C}^{2 } \otimes  \mathbb{C}^{2 }
)$,  we have
    \begin{equation*}\begin{split}\langle {\mathscr D}_1h,F\rangle_\varphi&=\sum_{B,A,A_2',\ldots, A_k'}\left ( Z^{ A_2'}_{[B} h_{A]A_2 '\ldots A_k'},      F_{ A_3 '\ldots A_k'B A}\right)_\varphi=\sum_{B,A,A_2',\ldots, A_k'}\left ( Z^{ A_2'}_B h_{AA_2 '\ldots A_k'},      f_{ A_3 '\ldots A_k'B A}\right)_\varphi  \\ &=
 \sum_{A,A_2',\ldots, A_k'}\left (  h_{A_2 '\ldots A_k'A },  \sum_B \delta_{ (A_2'}^B F_{A_3 '\ldots A_k') B A }\right)_\varphi
 +\mathscr  B_0'
 \end{split}  \end{equation*}by dropping  antisymmetrisation by (\ref{eq:antisym1}), applying  Stokes' formula
 (\ref{eq:Stokes-Z}) as above and using symmetrisation. The  formal adjoint operator ${\mathscr D}_1^*$ has the expression (\ref{eq:T*})   if we choose $h$ compactly supported.   Then $ \langle {\mathscr D}_1h,F\rangle_\varphi =\left \langle h  ,  {\mathscr D}_1^*F \right\rangle_\varphi$
if and only if the boundary term
\begin{equation*}\begin{split}
   \mathscr  B_0'&  =- \int_{\partial\Omega} \sum_{ A,A_2',\ldots }  h_{A_2 '\ldots A_k'A} \overline{\sum_{B  } Z_{( A_2'}^Br\cdot F_{A_3 '\ldots A_k' )B A  }} e^{- \varphi } dS=0
 \end{split}\end{equation*}
 for any $h$,   i.e. (\ref{eq:Neumann-1'}) holds.
   \end{proof}

\subsection{ $k$-plurisubharmonicity and $k$-pseudoconvexity}
\begin{prop}
  Suppose that $\chi$ is an increasing smooth convex function over $[0,\infty)$. Then for $\psi(x)=\chi(\varphi(x))$, we have
  \begin{equation}\label{eq:}
   \mathscr L_k(\psi;\xi)  \geq  \chi '(\varphi )  \mathscr L_k( \varphi;\xi)
  \end{equation}for
any $ \xi \in\odot^{k-1}\mathbb{C}^{2 }\otimes \mathbb{C}^{2 }$. In particular $\psi$ is $k$-plurisubharmonic if $\varphi$ is.
  \end{prop}
\begin{proof} For
any $(\xi_{ A_2 '\ldots A_k' A})\in\odot^{k-1}\mathbb{C}^{2 }\otimes \mathbb{C}^{2 }$, it is easy to see that
      \begin{equation*}\begin{split}\mathscr L_k(\psi;\xi)  =&-\chi'(\varphi(x))\sum_{A,B, A_1', \ldots,A_k'}    Z_{B}^{A_1'}Z_{ (A_1 ' }^{ A}\varphi\cdot \xi_{ A_2'\ldots  A_k') A}
  \overline{ \xi_{ A_2'\ldots A_k' B } }\\&- \chi''(\varphi(x)) \sum_{A,B, A_1', \ldots,A_k'}     Z_{B}^{A_1'}\varphi\cdot Z_{ (A_1 ' }^{ A} \varphi\cdot \xi_{ A_2'\ldots  A_k') A}
  \overline{ \xi_{ A_2'\ldots A_k' B } } \\=&\chi '(\varphi(x))  \mathscr L_k( \varphi;\xi)+  \chi''(\varphi(x)) \sum_{ A_1', \ldots,A_k'}  \left| \sum_B  { Z^{B}_{(A_1'}\varphi \cdot \xi_{ A_2'\ldots A_k') B } }  \right|^2
      \end{split}\end{equation*} by $-Z_{B}^{A '}\varphi=\overline{Z^{B}_{A '}\varphi}$ for real $\varphi$   by (\ref{eq:Z-raise}) and using  symmetrisation by Lemma \ref{lem:sym} (1).
    The result follows.
\end{proof}
  For a  $k$-pseudoconvex domain, we can choose a  defining function satisfying $|dr|=1$ on the boundary by the following proposition.
\begin{prop} \label{prop:k-pseudoconvexity}
  $k$-pseudoconvexity  of a domain is independent of the choice of  defining functions.
  \end{prop}
\begin{proof} Suppose that $r$ and $\widetilde{r}$ are both  defining functions of the domain $\Omega$. Then $\widetilde{r}(x)=  \mu(x)r(x)$ for some nonvanishing function $\mu>0$ near $\partial\Omega$. Note that $Z_{ A'}^A\widetilde{r}=\mu\cdot Z_{ A'}^A {r}$ on the boundary. It is obvious that for any $x\in \partial\Omega$ and $A_1'\ldots  A_k'=0',1'$, $ \xi \in\odot^{k-1}\mathbb{C}^{2 }\otimes \mathbb{C}^{2 }$ satisfies
$
   \sum_{A=0,1 } Z_{(A_1'}^A\widetilde{r}(x)\cdot \xi_{A_2'\ldots  A_k')  A} =0,
$ if and only if it satisfies
$
   \sum_{A=0,1 } Z_{(A_1'}^A {r}(x)\cdot \xi_{A_2'\ldots  A_k')  A}=0.
$   So the boundary condition (\ref{eq:pseudoconvex-tangent}) for $\xi$ is independent of the choice of   defining functions.
Then for $x\in \partial\Omega$ and $\xi $ satisfying  the boundary  condition (\ref{eq:pseudoconvex-tangent}), we have
 \begin{equation*}\begin{split} \mathscr L_k(\widetilde{r};\xi)(x)=& - \sum_{A,B, A_1', \ldots,A_k'} \sum_{s =1}^k    Z_{B}^{A_1'}Z_{  A_s ' }^{ A} (\mu(x) r(x))\cdot \xi_{  A_1'\ldots \widehat{A_s'}\ldots  A_k'  A}
  \overline{ \xi_{ A_2'\ldots A_k' B } }\\=& \mu(x)\mathscr L( {r};\xi)(x)+
r(x) \mathscr L(\mu;\xi)+\Sigma' +\Sigma''  \end{split}\end{equation*}(the second term    vanishes on the boundary) by  (\ref{eq:sym-1})
with\begin{equation*}\begin{split}\Sigma'& =-   k \sum_{ B, A_1', \ldots,A_k'}    Z_{B}^{A_1'}\mu(x)\cdot
  \overline{ \xi_{ A_2'\ldots A_k' B } } \cdot \sum_{A}{Z_{  (A_1 ' }^{ A} }  r(x)\xi_{   A_2'\ldots   A_k')  A}=0,
    \end{split}\end{equation*} by  the  condition (\ref{eq:pseudoconvex-tangent}) for $\xi$  on the boundary,
and by $Z^{A ' }_{B}=-\overline{Z_{A ' }^{B}}$
    \begin{equation*}\begin{split}  \Sigma'' & =k\sum_{  A_1', \ldots,A_k'}\overline{\sum_{ B }  Z^{B}_{A_1'} r(x) \xi_{ A_2'\ldots A_k' B } } \cdot  \sum_{ A } {Z_{ ( A_1 ' }^{ A} }\mu (x) \cdot
  \xi_{ A_2'\ldots    A_k')  A}\\&=k\sum_{  A_1', \ldots,A_k'} \overline{\sum_{B}  Z^{B}_{(A_1'} r(x) \xi_{ A_2'\ldots A_k') B } } \cdot  \sum_{A} {Z_{ ( A_1 ' }^{ A} }\mu (x) \cdot
  \xi_{ A_2'\ldots    A_k')  A}  =0,
    \end{split}\end{equation*}for $\xi $ satisfying  the boundary  condition (\ref{eq:pseudoconvex-tangent}), by using symmetrisation by Lemma \ref{lem:sym} (1).
       At last we get
    $
        \mathscr L_k(\widetilde{r};\xi)=\mu \mathscr L_k( {r};\xi)
   $ on the boundary for $\xi $ satisfying  the   condition (\ref{eq:pseudoconvex-tangent}). The result follows.
\end{proof}

 $k$-plurisubharmonic functions and $k$-pseudoconvex domains are abundant by the following examples.

\begin{prop}\label{prop:abundant} $
  r_1(x)=x_1^2+x_2^2$ and $r_2(x)=x_3^2+x_4^2
$
are both strictly $k$-plurisubharmonic,  and   $\chi_1 (r_1)$, $\chi_2 (r_2)$  and their sum are all (strictly) $k$-plurisubharmonic  for any increasing smooth (strictly) convex functions $\chi_1 $ and  $\chi_2 $ over $[0,\infty)$.
\end{prop}\begin{proof}
Note that
   \begin{equation*}
      \mathscr L_k( {r};\xi)(x)=  \sum_{A,B, A_1', \ldots,A_k'}      Z_{B}^{A_1'}\overline{Z^{  A_1 ' }_{ A}   r }\cdot \xi_{ A_2'\ldots A_k'  A}
  \overline{ \xi_{ A_2'\ldots A_k' B } }+\sum_{s =2}^k \sum_{A,B, A_1', \ldots,A_k'} Z_{B}^{A_1'}\overline{Z^{  A_s ' }_{ A}   r }\cdot \xi_{ A_1'\ldots \widehat{A_s'}\ldots  A_k'  A}
  \overline{ \xi_{ A_2'\ldots A_k' B } },
   \end{equation*} and
 \begin{equation}\label{eq:k-CF-raised} ( Z_{A}^{A'})=  \left(
                                      \begin{array}{rr}- \partial_{x_3} -\textbf{i}\partial_{x_4}
                                    & -\partial_{x_1} -\textbf{i}\partial_{x_2} \\ \partial_{x_1}-\textbf{i}\partial_{x_2}
                                  &   - \partial_{x_3}+\textbf{i}\partial_{x_4}
                                                                               \end{array}
                                    \right),
\end{equation}by
 (\ref{eq:k-CF}) ¡¡and¡¡¡¡ (\ref{eq:Z_A^A'}).
    So $ Z_{0}^{0'}$ and $ Z_{1}^{1'}$ are independent of $x_0$ and $x_1
$. According to $(A, B)=(0,0), (1,1), (1,0)$ and $ (0,1)$, we get
 \begin{equation*}\begin{split} \mathscr L_k(r_1;\xi)= & \quad \sum_{ A_2', \ldots,A_k' } Z_{0}^{1'}\overline{ Z_{0}^{1'}} r_1\cdot|\xi_{ A_2'  \ldots A_k'0 }  |^2+ \sum_{s=2}^k  \sum_{ A_2', \ldots,\widehat{A_s'}, \ldots } Z_{0}^{1'}\overline{ Z_{0}^{1'}} r_1\cdot\xi_{1'A_2'  \ldots \widehat{A_s'} \ldots A_k'0 }   \overline{\xi_{1'A_2'  \ldots \widehat{A_s'} \ldots A_k'0   }} \\
\quad& +\sum_{ A_2', \ldots,A_k' } Z_{1}^{0'}\overline{ Z_{1}^{0'}} r_1\cdot|\xi_{ A_2'  \ldots A_k'1 }  |^2+  \sum_{s=2}^k  \sum_{ A_2', \ldots,\widehat{A_s'}, \ldots } Z_{1}^{0'}\overline{ Z_{1}^{0'}}  r_1\cdot\xi_{0'A_2'  \ldots \widehat{A_s'} \ldots A_k'1 }   \overline{\xi_{0'A_2'  \ldots \widehat{A_s'} \ldots A_k'1   }}
  \\
\quad&+ \sum_{s=2}^k  \sum_{ A_2', \ldots,\widehat{A_s'}, \ldots } Z_{0}^{1'}\overline{ Z_{1}^{0'}} r_1\cdot \xi_{1'A_2'  \ldots \widehat{A_s'} \ldots A_k'1 }   \overline{\xi_{0'A_2'  \ldots \widehat{A_s'} \ldots A_k'0   }}\\
\quad&+ \sum_{s=2}^k  \sum_{ A_2', \ldots,\widehat{A_s'}, \ldots }  Z_{1}^{0'}\overline{ Z_{0}^{1'}}  r_1\cdot \xi_{0'A_2'  \ldots \widehat{A_s'} \ldots A_k'0 }   \overline{\xi_{1'A_2'  \ldots \widehat{A_s'} \ldots A_k'1   }}
   \\
= &  4 \sum_{  A,A_2',    \ldots,  {A_k'}  } | \xi_{ A_2'    \ldots A_k'A } |^2+4(k-1)\sum_{ B_3',  \ldots, B_k' }\left( | \xi_{ 1'B_3'  \ldots B_k'0   } |^2+|  \xi_{ 0'B_3'  \ldots B_k'1   } |^2\right) \geq4   | \xi  |^2
   \end{split}\end{equation*}
by
\begin{equation}\label{eq:ZZr'}
Z_{0}^{1'}\overline{ Z_{0}^{1'}}r_1=4=  Z_{1}^{ 0'}\overline{ Z_{1}^{0'}}r_1  \qquad {\rm and} \qquad Z_{1}^{0'}\overline{ Z_{0}^{1'} }r_1=0=Z^{1'}_{0}\overline{ Z_{1}^{0'}} r_1,
\end{equation}since
$
Z_{0}^{1'}\overline{ Z_{0}^{1'}}=\partial_{x_1}^2+\partial_{x_2}^2=Z_{1}^{0'}\overline{ Z_{1}^{0'}} $ and $   Z_{0}^{1'}\overline{ Z_{1}^{0'}}=-\partial_{x_1}^2+\partial_{x_2}^2 -2\mathbf{i}\partial_{x_1} \partial_{x_2}
$ by (\ref{eq:k-CF-raised}). Similarly $ Z_{0}^{1'}$ and $ Z_{1}^{0'}$ are independent of $x_3$ and $x_4
$ by (\ref{eq:k-CF-raised}), and so
 \begin{equation*}\begin{split} \mathscr L_k(r_2;\xi)= & \quad \sum_{ A_2', \ldots,A_k' } Z_{0}^{0'}\overline{ Z_{0}^{0'}} r_2\cdot|\xi_{ A_2'  \ldots A_k'0 }  |^2+ \sum_{s=2}^k  \sum_{ A_2', \ldots,\widehat{A_s'}, \ldots } Z_{0}^{0'}\overline{ Z_{0}^{0'}} r_2\cdot\xi_{0'A_2'  \ldots \widehat{A_s'} \ldots A_k'0 }   \overline{\xi_{0'A_2'  \ldots \widehat{A_s'} \ldots A_k'0   }} \\
\quad& +\sum_{ A_2', \ldots,A_k' } Z_{1}^{1'}\overline{ Z_{1}^{1'}} r_2\cdot|\xi_{ A_2'  \ldots A_k'1 }  |^2+  \sum_{s=2}^k  \sum_{ A_2', \ldots,\widehat{A_s'}, \ldots } Z_{1}^{1'}\overline{ Z_{1}^{1'}}  r_2\cdot\xi_{1'A_2'  \ldots \widehat{A_s'} \ldots A_k'1 }   \overline{\xi_{1'A_2'  \ldots \widehat{A_s'} \ldots A_k'1   }}
  \\
\quad&+ \sum_{s=2}^k  \sum_{ A_2', \ldots,\widehat{A_s'}, \ldots } Z_{0}^{0'}\overline{ Z_{1}^{1'}} r_2\cdot \xi_{0'A_2'  \ldots \widehat{A_s'} \ldots A_k'1 }   \overline{\xi_{1'A_2'  \ldots \widehat{A_s'} \ldots A_k'0   }}\\
\quad&+ \sum_{s=2}^k  \sum_{ A_2', \ldots,\widehat{A_s'}, \ldots }  Z_{1}^{1'}\overline{ Z_{0}^{0'}}  r_2\cdot \xi_{1'A_2'  \ldots \widehat{A_s'} \ldots A_k'0 }   \overline{\xi_{0'A_2'  \ldots \widehat{A_s'} \ldots A_k'1   }}
   \\
= &  4 \sum_{  A,A_2',    \ldots,  {A_k'}  } | \xi_{ A_2'    \ldots A_k'A } |^2+4(k-1)\sum_{ B_3',  \ldots, B_k' }\left( | \xi_{ 0'B_3'  \ldots B_k'0   } |^2+|  \xi_{ 1'B_3'  \ldots B_k'1   } |^2\right) \geq4   | \xi  |^2\end{split}\end{equation*}
by $
    Z_{0}^{0'}\overline{ Z_{0}^{0'}}r_2=4=Z_{1}^{1'}\overline{ Z_{1}^{1'}}r_2 $ and $Z_{1}^{1'}\overline{ Z_{0}^{0'}}r_2= 0=Z_{0}^{0'}\overline{ Z_{1}^{1'}}r_2,
$ since
$
    Z_{0}^{0'}\overline{ Z_{0}^{0'}}=\partial_{x_3}^2+\partial_{x_4}^2=Z_{1}^{1'}\overline{ Z_{1}^{1'}}$ and $ Z_{1}^{1'}\overline{ Z_{0}^{0'}}= \partial_{x_3}^2-\partial_{x_4}^2 -2\mathbf{i}\partial_{x_3} \partial_{x_4},
$ which follows from (\ref{eq:k-CF-raised}).
The result follows.
  \end{proof}

\begin{rem} The $k$-pseudoconvexity in  (\ref{eq:pseudoconvex}) is the natural convexity associated to the  $k$-Cauchy-Fueter complex (cf. H\"ormander \cite{Hor-convex} for    notions  of convexity associated to    differential operators).
The $k$-pseudoconvexity similarly defined in   $\mathbb{R}^{4n}$  is different from the   pseudoconvexity introduced in \cite{Wa11}, which is  based on the notion of   a plurisubharmonic function over quaternionic space introduced by Alesker  \cite{alesker1} (see also \cite{WaWang}).
   \end{rem}
    \section{The $L^2$-estimate  and the proof of the main theorem}
     \subsection{The $L^2$-estimate  }
      By the following density Lemma \ref{prop:domain-k} and Proposition \ref{prop:Neumann-k} (2), it is sufficient to show the   $L^2$ estimate (\ref{eq:L2-k}) for $f\in C ^\infty (\overline{ \Omega},   \odot^{k-1}
 \mathbb{C}^{2 } \otimes\mathbb{C}^{2 })$ satisfying the boundary condition (\ref{eq:boundary-condition}). By expanding symmetrisation in terms of (\ref{eq:sym-1}) and using commutators,   we get
   \begin{equation}\label{eq:sigma-0-k}\begin{split}
k\langle {\mathscr D}_0^*f, {\mathscr D}_0^*f\rangle_{\varphi } &  =  k\langle {\mathscr D}_0 {\mathscr D}_0^*f,f\rangle_{\varphi } = k \sum_{ B, A_1', \ldots,A_k'}\left(Z_{B}^{A_1'}\sum_{ A }\delta_{(A_1 ' }^{ A}  f_{ A_2'\ldots A_k')A   } ,  f_{ B   A_2'\ldots A_k' }\right)_{\varphi }\\& = \sum_{A,B, A_1', \ldots,A_k'}  \left( Z_{B}^{A_1'} \delta_{ A_1 ' }^{ A} f_{ A_2'\ldots  A_k' A}  +  \sum_{s=2}^k  Z_{B}^{A_1'}\delta_{ A_s ' }^{ A} f_{   A_1'\ldots \widehat{A_s'} \ldots A_k' A } ,
   f_{   A_2'\ldots A_k' B}\right)_{\varphi } \\& = \sum_{A,B, A_1', \ldots,A_k'}\left\{ \left( \delta_{ A_1 ' }^{ A} Z_{B}^{A_1'} f_{ A_2'\ldots  A_k' A  }  +   \sum_{s=2}^k  \delta_{ A_s ' }^{ A} Z_{B}^{A_1'} f_{   A_1'\ldots \widehat{A_s'} \ldots A_k' A } ,
   f_{ A_2'\ldots A_k' B  }\right)_{\varphi }\right.\\&\quad \qquad \qquad  +  \left.\left( \left[ Z_{B}^{A_1'}, \delta_{ A_1 ' }^{ A}\right] f_{ A_2'\ldots  A_k' A }  +   \sum_{s=2}^k \left[Z_{B}^{A_1'},\delta_{ A_s ' }^{ A}\right] f_{   A_1'\ldots \widehat{A_s'} \ldots A_k' A   } ,
   f_{ A_2'\ldots A_k'B }\right )_{\varphi }\right\}  \end{split} \end{equation}
   Using Stokes' formula (\ref{eq:Stokes-delta}) and commutators
  \begin{equation}\label{eq:commutators}
     \left [Z^{ A ' }_{ B},\delta_{B ' }^{ A}\right ]=-  Z^{ A ' }_{ B}Z_{B ' }^{ A}\varphi
  \end{equation}
       by $[Z^{ A ' }_{ A},Z_{B ' }^{ A} ]=0$  (since they are of constant coefficients), we get
    \begin{equation}\label{eq:sigma-0-k'}\begin{split}k\langle {\mathscr D}_0^*f, {\mathscr D}_0^*f\rangle_{\varphi }& = \sum_{A,B, A_1', \ldots,A_k'}  \left( Z_{B}^{A_1'} f_{  A_2'\ldots  A_k' A  } ,Z^{ A_1 ' }_{ A}
   f_{ A_2'\ldots A_k' B }\right)_{\varphi }\\&\qquad +   \sum_{A,B, A_1', \ldots,A_k'}\sum_{s=2}^k  \left(  Z_{B}^{A_1'} f_{   A_1'\ldots \widehat{A_s'} \ldots A_k' A } ,Z^{ A_s ' }_{ A}
   f_{ A_2'\ldots A_k' B }\right)_{\varphi } +\mathscr B_k +\mathscr C_k\\&  =:\Sigma_1 +\Sigma_2+\mathscr B_k +\mathscr C_k \end{split}\end{equation}
where the commutator term is
 \begin{equation}\label{eq:C_k}\begin{split}
    \mathscr C_k:&=- \sum_{A,B, A_1', \ldots,A_k'} \left(   Z_{B}^{A_1'}Z_{ A_1 ' }^{ A} \varphi(x)\cdot f_{ A_2'\ldots  A_k' A }   +  \sum_{s=2}^k   Z_{B}^{A_1'}Z_{ A_s ' }^{ A}\varphi(x)\cdot f_{     A_1'\ldots \widehat{A_s'} \ldots  A_k' A } ,
   f_{  A_2'\ldots A_k' B }\right)_{\varphi }\\&
    =  \int_{ \Omega}  \mathscr L_k( \varphi;f(x))(x)e^{- \varphi(x)}dV \geq   c\|f\|_{\varphi }^2
\end{split} \end{equation}by
 (\ref{eq:commutators}) and   assumption (\ref{eq:Levi0}),
and the boundary term is
 \begin{equation}\label{eq:B_k}
    \mathscr B_k:=  \sum_{A,B, A_1', \ldots,A_k'} \int_{\partial\Omega}  \sum_{s=1}^k  Z_{ A_s ' }^{ A}r \cdot Z_{B}^{A_1'}f_{    A_1'\ldots \widehat{A_s'} \ldots    A_k' A}\cdot
  \overline{ f_{  A_2'\ldots A_k' B }} e^{-{\varphi }}dS.
 \end{equation}This boundary term    can also be handled by Morrey's technique. Since $\sum_{A=0,1 } Z_{ (A_1'}^Ar \cdot f_{A_2 '\ldots A_k')   A  }=0 $ vanishing on the boundary  for fixed $ {A_1 '\ldots A_k'}=0',1'$, there exists $(k+1)$  functions  $\lambda_{A_1 '\ldots A_k'}=\lambda_{(A_1 '\ldots A_k')}$ such that for $x$ near $ \partial\Omega$,
 \begin{equation*}
    k  \sum_{A=0,1 } Z_{ (A_1'}^Ar(x)\cdot f_{A_2 '\ldots A_k')   A  }(x)=\lambda_{A_1 '\ldots A_k'}(x) \cdot r(x).
 \end{equation*}
   Now differentiate  this equation
by the complex vector field $Z^{ A_1 ' }_{ B}$ to get
 \begin{equation*}
  \sum_{A=0,1 }\left\{k Z^{ A_1 ' }_{ B} Z_{ (A_1'}^Ar\cdot f_{A_2 '\ldots A_k')   A  }+   \sum_{s=1}^k Z_{  A_s'}^Ar\cdot Z^{ A_1 ' }_{ B }f_{A_1' \ldots\widehat{A_s'} \ldots A_k'    A  }\right\}=Z^{ A_1 ' }_{ B}\lambda_{A_1 '\ldots A_k'} \cdot r+\lambda_{A_1 '\ldots A_k'} Z^{ A_1 ' }_{ B} r.
   \end{equation*}
   Then multiplying it by $\overline{   f_{A_2 '\ldots A_k' B     }}$ and taking summation over $ B,A_1 ',\ldots, A_k'$, we get that
\begin{equation}\label{eq:boundary-id}\begin{split}
    & -   \mathscr L_k(r;f(x))(x)+  \sum_{A,B,A_1 ',\ldots, A_k'} \sum_{s=1}^k Z_{  A_s'}^Ar(x)\cdot Z^{ A_1 ' }_{ B }f_{A_1' \ldots\widehat{A_s'}\ldots A_k'    A  } \overline{   f_{A_2 '\ldots A_k' B     }}   \\=&-
          \sum_{A_1 ',\ldots, A_k'}\left(\lambda_{A_1 '\ldots A_k'}(x) \cdot  \sum_{B }    \overline{Z_{( A_1 ' }^{ B} r(x) \cdot  f_{A_2 '\ldots A_k') B     }}\right)=0 \end{split}\end{equation} on the boundary $\partial \Omega$,
by using $r(x)|_{\partial \Omega}=0$, symmetrisation  by (\ref{eq:bracket-omit}) and  the boundary condition  (\ref{eq:boundary-condition}) for  $f$, where $\lambda$ is symmetric in the primed indices. Apply (\ref{eq:boundary-id}) to the boundary term  (\ref{eq:B_k}) to get
   \begin{equation}\label{eq:B-k}
\mathscr B_k
    =  \int_{\partial\Omega} \mathscr L_k(r;f(x)) e^{-\varphi}dS
    \geq0
  \end{equation}
by the pseudoconvexity (\ref{eq:pseudoconvex})-(\ref{eq:pseudoconvex-tangent}) of $r$ and $f$ satisfying the boundary condition (\ref{eq:boundary-condition}).

Now for the second sum of (\ref{eq:sigma-0-k'}), we have
\begin{equation}\label{eq:sigma-2-k'}\begin{split}
 \Sigma_2&   =(k-1) \sum_{A,B,  B _3', \ldots, B _k'}  \left( \sum_{A '} Z^{ A ' }_{  B } f_{ A      A '{ B _3'} \ldots   B _k' } , \sum_{B '} Z^{ B ' }_{ A} f_{  B    B'{ B _3'}\ldots    B _k'  }\right)_{\varphi } \\&    =(k-1)\sum_{A,B, { B _3'}, \ldots ,  B _k'} \left\{\left\|\sum_{A'} Z^{ A ' }_{  A} f_{B      A' { B _3'} \ldots   B _k' }\right\|^2_{\varphi }-2   \left\|\sum_{A'} Z^{ A ' }_{ [ A } f_{ B]  A'  { B _3'} \ldots   B _k' }  \right\|^2_{\varphi }\right\}\\&
    =(k-1)\sum_{A,B, { B _3'}, \ldots,   B _k'} \left\|\sum_{A'}  Z^{ A ' }_{  A} f_{B    A'  { B _3'} \ldots   B _k' }\right\|^2_{\varphi }- 2 (k-1) \left\| {\mathscr D}_1  f  \right\|^2_{\varphi }
   \end{split}\end{equation}by relabeling indices and applying Lemma \ref{lem:sym} (3).
By applying Lemma \ref{lem:sym} (3) again and using  $Z_A^{0'}= Z_{1'A}$ and $ Z_A^{1'}=-Z_{0'A}$ in (\ref{eq:Z_A^A'}), we get
\begin{equation}\label{eq:sigma-1'-k}\begin{split}
   \Sigma_1 &= \sum_{ A,B, A_1', \ldots,A_k'} \left\{\left\|  Z^{ A_1 ' }_{ A} f_{  B  A_2'\ldots  A_k' }  \right\|_{\varphi }^2-2  \left\|  Z^{ A_1 ' }_{ [A} f_{  B]  A_2'\ldots  A_k' }  \right\|_{\varphi }^2\right\}\\&  = \sum_{ A,B, A_1', \ldots,A_k'} \left\|  Z^{ A_1 ' }_{ A} f_{  B  A_2'\ldots  A_k' }  \right\|_{\varphi }^2-4\sum_{  A_1', \ldots,A_k'} \left\|  Z_{ A_1 '   [0} f_{  1]  A_2'\ldots  A_k' }  \right\|_{\varphi }^2.
\end{split} \end{equation}
Substituting  (\ref{eq:C_k}) and (\ref{eq:B-k})-(\ref{eq:sigma-1'-k}) into (\ref{eq:sigma-0-k'}), we get the estimate
   \begin{equation}\label{eq:estimate-k-0}
     k \|{\mathscr D}_0^*f\|^2_{\varphi }+ 2 (k-1)    \left\| {\mathscr D}_1  f  \right\|^2_{\varphi }\geq c \left \| f \right\|^2_{\varphi } -\sum_{  A_1', \ldots,A_k'} \left\| 2 Z_{ A_1 '   [0} f_{  1]  A_2'\ldots  A_k' }  \right\|_{\varphi }^2.
   \end{equation}
When $k=2$, the term $\sum_{ A',B'}\left\|2 Z_{ A '  [0} f_{ 1]   B' } \right\|^2_{\varphi }$ is controlled by $4( \|{\mathscr D}_0^*f\|^2_\varphi+      \|d \varphi\|_\infty^2 \cdot \left\|f   \right\|^2_{\varphi }+  \|{\mathscr D}_1f\|^2_\varphi)$ simply by the identity
\begin{equation}\label{eq:Sigma-1-trace}\begin{split}
 - 2 Z_{ A '  [0} f_{ 1]   B' }   &= \sum_{A=0,1}Z_{ A '}^{A} f_{ B' A  } = \sum_A Z_{( A '  }^{A} f_{ B') A } + \sum_{A }Z_{ [A '}^{A} f_{ B'] A }  ,
 \end{split}\end{equation} where the first sum is $ ({\mathscr D}_0^*f)_{A '    B'}+ \sum_A Z_{( A '  }^{A}\varphi\cdot f_{  B')A }$, while the second sum is $\sum_{A=0,1}Z_{ [0 '}^{A} f_{ 1'] A }
   =  \sum_{A'=0',1'}Z_{ [0 }^{A'} f_{1]  A '  }=  ({\mathscr D}_1f)_{01}$.

To estimate the last term in (\ref{eq:estimate-k-0}) for general $k$,  fix  $A_1',\ldots,  A_k'$.
{\it Case  i:     $A_1'+\ldots+  A_k'=l\neq0, k$
 and $ A_1 '=0' $}.   It follows from (\ref{eq:sym-1}) that
 \begin{equation}\label{eq:sym-k}\begin{split}
   \sum_{A }  Z_{ (0 '}^A f_{  0'\ldots 0'\underbrace{\scriptstyle1'  \ldots 1'}_l)  A }&=\frac {k-l}k  \sum_{A } Z_{ 0 '    }^A f_{     0'\ldots 0'\underbrace{\scriptstyle1'  \ldots 1'}_l A}+ \frac { l}k  \sum_{A } Z_{ 1 '    }^A f_{  0'\ldots 0'\underbrace{\scriptstyle 1'  \ldots 1'}_{l-1} A} \\&=-\frac {2(k-l)}k Z_{ 0 '   [0} f_{  1]  0'\ldots 0'\underbrace{\scriptstyle 1' \ldots 1' }_{l }} - \frac {2l}k  Z_{ 1'   [0} f_{  1]  0'\ldots 0'\underbrace{\scriptstyle 1' \ldots 1'}_{l-1}}
   \end{split} \end{equation}
 by $f$ symmetric in the primed indices and using Lemma \ref{lem:raise}.
 Then
 \begin{equation}\label{eq:term-type-1}\begin{split}
    2 Z_{ A_1 '   [0} f_{  1]  A_2'\ldots  A_k' }= &2 Z_{ 0 '   [0} f_{  1]  0'\ldots 0'\underbrace{\scriptstyle 1'  \ldots 1'}_l  }\\
     =&\frac {2(k-l)}k  Z_{ 0 '   [0} f_{  1]  0'\ldots 0'\underbrace{\scriptstyle1'  \ldots 1'}_l }+\frac {2 l}k  Z_{ 1'   [0} f_{  1]  0'\ldots 0' \underbrace{\scriptstyle1' \ldots 1' }_{l-1}}   \\&\hskip 5mm+\frac {2l}k   Z_{ 0 '   [0} f_{  1]  0'\ldots 0'\underbrace{\scriptstyle 1'  \ldots 1'}_l }-\frac {2 l}k  Z_{ 1'   [0} f_{  1]  0'\ldots 0' \underbrace{\scriptstyle1' \ldots 1'}_{l-1} }
     \\=&- \sum_{A }Z_{ (0 '}^A f_{  0'\ldots 0'\underbrace{\scriptstyle 1'  \ldots 1'}_l )  A } -\frac {2l}k\sum_{A'} Z^{A'}_{    [0} f_{  1] A'  0'\ldots 0' \underbrace{\scriptstyle1' \ldots 1'}_{l-1}} \\=&- (\mathscr D_0^*f)_{ A_1'\ldots  A_k' }-  \sum_{A }Z_{ (A_1' }^A\varphi\cdot  f_{  A_2'\ldots  A_k')  A }  - \frac {2l}{ k}(\mathscr D_1 f)_{   01 0'\ldots 0' \underbrace{\scriptstyle1' \ldots 1'}_{l-1}}
   \end{split} \end{equation} by using Lemma  \ref{lem:raise} again,  (\ref{eq:sym-k}) and $f$ symmetric in the primed indices.

  {\it Case ii:     $A_1'+\ldots+  A_k'=k-l\neq0, k$  and $ A_1 '=1' $
 }.   We have the similar identity     by
       \begin{equation}\label{eq:term-type-2}\begin{split}
    2 Z_{ A_1 '   [0} f_{  1]  A_2'\ldots  A_k' }= &2Z_{ 1 '   [0} f_{  1]  1'\ldots 1' \underbrace{\scriptstyle 0'  \ldots 0'}_{l }}\\
     =&\frac {2 (k-l) }k  Z_{ 1 '   [0} f_{  1]  1'\ldots 1' \underbrace{\scriptstyle 0'  \ldots 0'}_{l } }+\frac {2l}k  Z_{ 0'   [0} f_{  1]  1'\ldots 1'\underbrace{\scriptstyle 0'  \ldots 0'}_{l-1} } \\&\hskip 5mm+\frac {2l}k  Z_{ 1'   [0} f_{  1]  1'\ldots 1' \underbrace{\scriptstyle 0' \ldots 0'}_{l} }-\frac {2l}k  Z_{ 0'   [0} f_{  1]  1'\ldots 1' \underbrace{\scriptstyle 0'  \ldots 0'}_{l-1 } } \\=&- (\mathscr D_0^*f)_{ A_1'\ldots  A_k' }-  \sum_{A }Z_{ (A_1' }^A\varphi\cdot  f_{  A_2'\ldots  A_k')  A }  + \frac {2l}{ k}(\mathscr D_1 f)_{   01  1'\ldots 1' \underbrace{\scriptstyle 0'  \ldots 0'}_{l-1 } } .
         \end{split} \end{equation}
     {\it Case iii:  $A_1'=\ldots= A_k'=0'$ or $ 1'$}. We have
   \begin{equation}\label{eq:term-type-3}\begin{split}2Z_{ 0 '   [0} f_{  1]  0'\ldots 0'  }
 = -\sum_{A }  Z_{ (0 '}^A f_{  0'\ldots 0' )  A } =-(\mathscr D_0^*f)_{ 0'\ldots  0' }-  \sum_{A }Z_{ (0' }^A\varphi\cdot  f_{  0'\ldots  0')  A },
   \end{split} \end{equation}and similar identity holds for $A_1'=\ldots=  A_k'=1'$.

     We can use $|(a_1+\ldots a_k)/k|^2\leq (|a_1|^2+\ldots |a_k|^2)/k $ and the Cauchy-Schwarz inequality  to control the norm of  $\sum_{A }Z_{ (A_1' }^A\varphi\cdot  f_{  A_2'\ldots  A_k')  A } $ by $2\|d\varphi\|_\infty^2   \left \| f \right\|^2_{\varphi } $
   since
          \begin{equation}\label{eq:|d-varphi|}
               |d \varphi |^2 = \sum_A \left|Z_{0 '  }^{A}\varphi\right|^2=\sum_A \left|Z_{1 '  }^{A}\varphi\right|^2
          \end{equation} by $\varphi$ real and
 \begin{equation}\label{eq:k-CF-raised-2} \left( Z^{A}_{A'}\right)=  \left(
                                      \begin{array}{rr}  \partial_{x_3} -\textbf{i}\partial_{x_4}
                                    & -\partial_{x_1} -\textbf{i}\partial_{x_2} \\ \partial_{x_1}-\textbf{i}\partial_{x_2}
                                  &     \partial_{x_3}+\textbf{i}\partial_{x_4}
                                                                               \end{array}
                                    \right),
\end{equation}
          by   $Z_{A'}^A$ in (\ref{eq:Z_A^A'}) and  (\ref{eq:k-CF}).   Note that the term $ \|2 Z_{ 0 '   [0} f_{  1]  0'\ldots 0'\underbrace{\scriptstyle 1'  \ldots 1'}_l  }  \|_{\varphi }^2$ appears $C_{k-1}^l$ times in the summation
   \begin{equation*}
       \sum_{  A_1', \ldots,A_k'} \left\| 2 Z_{ A_1 '   [0} f_{  1]  A_2'\ldots  A_k' }   \right\|_{\varphi }^2,
   \end{equation*} while $ \| (\mathscr D_1 f)_{   01 1'\ldots 1'\underbrace{\scriptstyle 0'  \ldots 0'}_{l-1} } \|_{\varphi }^2$ appears $C_{k-2}^{l-1}$ times in the definition of $\left\| \mathscr D_1 f  \right\|_{\varphi }^2$. It is similar for terms in the case {\it ii}.
      Then by using $|a+b+c|^2\leq  4|a|^2+2|b|^2+4|c|^2$ and $\frac {l^2}{k^2 }C_{k-1}^l\leq C_{k-2}^{l-1}$, we get from (\ref{eq:term-type-1})-(\ref{eq:term-type-3}) that
   \begin{equation}\label{eq:sigma-1}\begin{split}
    \sum_{  A_1', \ldots,A_k'} \left\| 2 Z_{ A_1 '   [0} f_{  1]  A_2'\ldots  A_k' }   \right\|_{\varphi }^2 &\leq  4\left\| \mathscr D_0^*f  \right\|_{\varphi }^2+ 4\|d\varphi\|_\infty^2   \left \| f \right\|^2_{\varphi } +  16\sum_{l=1}^{k-1}C_{k-2}^{l-1}\| (\mathscr D_1 f)_{   01 0'  \ldots 0'\underbrace{\scriptstyle 1'\ldots 1'}_{l-1} } \|_{\varphi }^2\\&\qquad\qquad\qquad\qquad\qquad\qquad\quad +16\sum_{l=1}^{k-1}C_{k-2}^{l-1}\| (\mathscr D_1 f)_{   01 1'\ldots 1'\underbrace{\scriptstyle 0'  \ldots 0'}_{l-1} } \|_{\varphi }^2
     \\&= 4\left\| \mathscr D_0^*f  \right\|_{\varphi }^2+ 4\|d\varphi\|_\infty^2   \left \| f \right\|^2_{\varphi } +16 \left\| \mathscr D_1 f  \right\|_{\varphi }^2.
\end{split} \end{equation}
   Now substitute  (\ref{eq:sigma-1}) into (\ref{eq:estimate-k-0}) to get the estimate
   \begin{equation*}
     (k+4) \|{\mathscr D}_0^*f\|^2_{\varphi }+ (2  k+14)  \left\| {\mathscr D}_1  f  \right\|^2_{\varphi }\geq (c- 4\|d\varphi\|_\infty^2 ) \left \| f \right\|^2_{\varphi }.
   \end{equation*}
The estimate (\ref{eq:L2-k}) is proved.\hskip 109mm $\Box$

  \begin{rem} \label{eq:C0-kappa} (1)     For the weight $ \kappa\varphi$, we have
      $
  \mathscr L_k(\kappa\varphi;\xi) =  \kappa \mathscr L_k(\varphi;\xi) \geq c\kappa  |\xi |^2
    $
      for $\xi\in \odot^{k-1}\mathbb{C}^{2 }\otimes \mathbb{C}^{2 }$, and  $c\kappa-4\|d (\kappa\varphi)\|_\infty^2=\kappa (c-4 \kappa\|d\varphi\|_\infty^2)>0$ if $0<\kappa<\frac c{4 \|d\varphi\|_\infty^2}$.

(2) On    $\mathbb{R}^{4n}$ with $n>1$, the negative term in $ \Sigma_1$ in (\ref{eq:sigma-1'-k}) becomes
$
   -\sum_{A,B=0}^{2n-1} \sum_{   A_1', \ldots,A_k'}   \|  Z^{ A_1 ' }_{ [A} f_{  B]  A_2'\ldots  A_k' }  \|_{\varphi }^2,
$
 which can not be simply estimated by using (\ref{eq:sym-k}). But over the whole space $\mathbb{R}^{4n}$ with weight $\varphi=|x|^2$, if we do not handle the nonnegative term  $  \sum_{A,B, A_1', \ldots,A_k'}  ( Z_{B}^{A_1'} \delta_{ A_1 ' }^{ A} f_{ A_2'\ldots  A_k' A}    ,
   f_{   A_2'\ldots A_k' B} )_{\varphi } $ in (\ref{eq:sigma-0-k}) by using    commutators and   drop it directly,  we can obtain a weak    $L^2$ estimate  (cf. \cite{Wa-weighted}).
\end{rem}

 \subsection{The Density Lemma }The following density lemma can be deduced from a general result due to H\"ormander \cite{Hor} (see also \cite{SCZ}).

   \begin{lem}\label{prop:domain-k}
$ C^\infty \left(\overline{\Omega}, \odot^{k-1
}\mathbb{C}^{2 } \otimes \mathbb{C}^{2 } \right)\cap {\rm Dom} ({\mathscr D}_0^*)  $ is dense in ${\rm Dom} ({\mathscr D}_0^*) \cap  {\rm Dom}  ({\mathscr D}_1)$.
\end{lem}
In general, let $\Omega\subset \mathbb{R}^N$ be an
open set  and let
\begin{equation}\label{eq:L'}
   L(x,\partial) = \sum_{j=1}^N L_j(x)\partial_{x_j} + L_0(x)
\end{equation}
 be a first-order differential operator,  where
$L_j\in C^1(\Omega, \mathbb{C}^{K\times J})$, $j =  1 ,\cdots, N$, $L_0\in C^0(\Omega, \mathbb{C}^{K\times J})$. Here $\mathbb{C}^{K\times J} $ denotes the space of
$K \times J$ matrices with complex coefficients.   The graph of the {\it maximal differential
operator} defined by $L(x, \partial)$ in $L^2$ consists of all pairs $(f, u) \in L^2(\Omega, \mathbb{C}^{  J}) \times L^2(\Omega, \mathbb{C}^{K })$ such
that $L(x, \partial) f =u$ in the sense of distributions, i.e.
\begin{equation}\label{eq:unweighted}
 \langle f,L^*(x,\partial) v \rangle=\langle u,   v\rangle, \qquad{\rm for} \quad{\rm  any }\quad v\in C_0^\infty (\Omega, \mathbb{C}^{K })
\end{equation}where
$
   L^*(x,\partial) =- \sum L_j^*(x)\partial_{x_j} + L_0^*(x)-\sum \partial_{x_j}L_j^*(x)
$
is the formal adjoint operator of $L(x, \partial)$. Thus the   maximal differential
operator defined by $L(x, \partial)$ is the  adjoint of the differential operator $L^*(x, \partial)$ with domain $C_0^\infty (\Omega, \mathbb{C}^{K })$. So the maximal operator is closed and
its adjoint is the closure of $L^*(x, \partial)$, first defined with domain $C_0^\infty (\Omega, \mathbb{C}^{K })$. It is called the
{\it minimal operator} defined by $L^*(x, \partial)$. Similarly, the adjoint of the maximal operator defined
by $L^*(x, \partial)$ is the minimal operator defined by $L (x, \partial)$.

\begin{prop}\label{prop:minimal-domain} (proposition A.1 of \cite{Hor})
   If $\Omega$ is a bounded domain with $C^1$ boundary, $L_j\in C^1$ and $L_0\in C^0$ in a
neighborhood of $\overline{\Omega}$, then the maximal operator defined by $ L ( x, \partial)$ in (\ref{eq:L'}) is the closure of its restriction
to functions which are $C^\infty$ in a neighborhood of $\overline{\Omega}$. The minimal domain of $ L ( x, \partial)$ consists
precisely of the functions $f\in L^2(\Omega, \mathbb{C}^{  J})$ such that $L ( x , \partial ) \widetilde{f}\in L^2(\mathbb{R}^N, \mathbb{C}^{  K})$ if $\widetilde{f}: = f$ in $\Omega$ and
$\widetilde{f}:= 0$ in $\Omega^c$.
\end{prop}

Consider another  first-order differential operator
$
    M(x,\partial) = \sum M_j(x)\partial_{x_j} + M_0(x)
$
   where
$M_j\in C^1(U$, $ \mathbb{C}^{K'\times J})$, $j = 1 ,\cdots, N$, $M_0\in C^0(U$, $ \mathbb{C}^{K'\times J})$. For    an open neighborhood $U$  of  $0\in \mathbb{R}^N$, denote $U_-:=\{x\in U; x_N<0\}$.
\begin{prop}\label{prop:minimal-domain2} (proposition A.2 of \cite{Hor})
    Assume that $\ker L_N(x)$ and $\ker L_N(x)\cap \ker M_N(x)$ have constant dimension
when $x \in U$. If $u\in L^2(U_-, \mathbb{C}^{  J})$ is in the minimal domain of $ L ( x, \partial)$ and the maximal
domain of $M(x, \partial)$, and if supp $f $ is sufficiently close to the origin, then there exists a
sequence $f_\nu\in  C_0^\infty (U)$ such that $f_\nu$ restricted to $U_-$ is in the minimal domain of $ L ( x, \partial)$ and
$f_\nu\rightarrow f$, $ L ( x, \partial)f_\nu\rightarrow  L ( x, \partial)f$, $ M ( x, \partial)f_\nu\rightarrow  M ( x, \partial)f$ in $L^2(U_-)$.
\end{prop}

 {\it Proof of Lemma \ref{prop:domain-k}}.  We denote by $ \mathcal{L}$ the differential operator given by the formal adjoint operator (\ref{eq:T*}) of ${\mathscr D}_0$. If we use notations in (\ref{eq:Dk}) to identify linear spaces in (\ref{eq:isomorphisms}), $\mathscr D_0$, $\mathscr D_1$ and $ \mathcal{L}$ are   $(2k)\times(k+1)$-,  $(k-1)\times (2k)$- and $(k+1)\times (2k)$-matrix valued differential operators of first order, respectively. In particular,
  \begin{equation}\label{eq:D0-k-nu-conj}
   \mathcal{ L}f= [\widehat{\mathcal{L}}-\widehat{\mathcal{L}}\varphi]f,\qquad \widehat{\mathcal{L}}=
   { \left(\begin{smallmatrix}Z_{0'}^{0}&
Z_{0'}^{1}
 & 0\hskip 4mm&0\hskip 4mm&0\hskip  4mm&0 \hskip  4mm&0\hskip  4mm&\cdots\\ \frac 1k Z_{1'}^{0}&\frac 1k Z_{1'}^1&\frac {k-1}k Z_{0'}^0&
\frac {k-1}k Z_{0'}^1
 &0\hskip  4mm&0\hskip  4mm &0\hskip  4mm&\cdots\\
0\hskip  3mm&0\hskip  3mm&\frac {2}k Z_{1'}^{0}&\frac {2}k  Z_{1'}^1&\frac {k-2}k  Z_{0'}^0&
\frac {k-2}k Z_{0'}^1&0 \hskip  4mm& \cdots \\
0\hskip  3mm&0\hskip  3mm& 0\hskip  4mm&0\hskip  4mm&\frac {3}k Z_{1'}^{0}&\frac {3}k  Z_{1'}^1&\frac {k-3}k  Z_{0'}^0&
 \cdots\\\vdots \hskip 3mm&\vdots\hskip 3mm &\vdots\hskip 4mm &  \vdots\hskip 4mm & \vdots\hskip 4mm & \vdots\hskip 4mm &\vdots\hskip 4mm&\vdots\hskip 4mm
\end{smallmatrix}\right)_{(k+1)\times (2k)}}
\end{equation}by (\ref{eq:sym-k}),   and $\mathcal{M}:=\mathscr D_1$ with
\begin{equation}\label{eq:mathcal-M-k}\begin{split}
    {\mathcal{M}}& = \frac 12{ \left(\begin{smallmatrix}-Z_1^{0'}&Z_0^{0'}&-Z_1^{1'}&Z_0^{1'}&0&0&0&\cdots\\
    0&0& -Z_1^{0'}&Z_0^{0'}&-Z_1^{1'}&Z_0^{1'} &0&\cdots\\
    0&0&0&0& -Z_1^{0'}&Z_0^{0'} &- {Z}_1^{1'}&\cdots\\ \vdots &\vdots &\vdots &\vdots &\vdots &\vdots&\vdots &\vdots
\end{smallmatrix}\right)_{(k-1)\times (2k)}} .  \end{split}\end{equation} Note that here $\mathcal{ L}$ as the formal adjoint operator of $\mathscr D_0 $ is different from that in \cite{CMW} \cite{Wa10}, because the inner product     (\ref{eq:inn-product2}) we use here is  different from that in \cite{CMW} \cite{WR} \cite{Wa10} even when $\varphi=0$.
     Since $Z_A^{A'}$'s and $Z^A_{A'}$'s are complex vector fields with constant coefficients,  we can write
\begin{equation}\label{eq:Lj-Mj}
 \mathcal{L}=\sum_{j=1}^4\mathcal{ L}_j\frac \partial{\partial x_j}+\mathcal{ L}_0, \qquad \mathcal{ L}_0=-\widehat{\mathcal{L}}\varphi, \qquad\mathcal{ M}=\sum_{j=1}^4\mathcal{M}_j\frac \partial{\partial x_j} ,
\end{equation}
where $\mathcal{ L}_j$'s are constant   $(k+1)\times (2k)$-matrices and $ \mathcal{ M}_j$'s are constant $(k-1)\times (2k)$-matrices.

 By definition,  ${\rm Dom}(\mathscr D_j)$  is exactly the domain  of the maximal  operator defined by the differential operator $\mathscr D_j$, $j=0,1$. To apply the above propositions, let us show that ${\rm Dom}(\mathscr D_0^*)$ coincides with the domain of the minimal operator defined by $\mathcal{L}$. Recall that $f\in {\rm Dom}(\mathscr D_0^*)$  if and only if there exists  some $u\in  L^2(\Omega , \mathbb{C}^{k+1})$ such that  $\langle \mathscr D_0 v, f\rangle_\varphi=\langle v, u\rangle_\varphi$ for any $v\in  {\rm Dom}(\mathscr D_0 )$. In terms of the unweighted inner product $\langle\cdot,\cdot\rangle$ in (\ref{eq:unweighted}), it is equivalent to
  \begin{equation*}
     \langle (\mathscr D_0+¡¡\mathscr D_0\varphi) v, f\rangle =\langle v, u\rangle\end{equation*}for any $v\in  {\rm Dom}(\mathscr D_0 )$,
 since the weight function $\varphi$ is smooth on the bounded domain $\Omega$. Note that the minimal operator defined by $\mathcal{ L}$  in (\ref{eq:D0-k-nu-conj}) is  the   adjoint operator
of the maximal operator defined by $\mathscr D_0+¡¡\mathscr D_0\varphi$ with respect to the unweighted inner product (\ref{eq:unweighted}). so $f$ is in the domain of the minimal operator defined by $\mathcal{L}$. The converse is also true.
 By abuse of notations, we denote the minimal operator defined by $\mathcal{ L}$ also by $\mathcal{ L}$ and the maximal operator defined by $ \mathcal{M}$  also by  $ \mathcal{M}$.

 Suppose that $f\in {\rm Dom} ({\mathscr D}_0^*) \cap  {\rm Dom}  ({\mathscr D}_1)$. Write $ \mathcal{ L} f=u$.  By Proposition \ref{prop:minimal-domain}, $f$   in the  minimal domain of $\mathcal{ L}$ implies that
   \begin{equation*}
     \mathcal{ L}\widetilde{f}=\widetilde{u} \end{equation*}
as $L^2(\mathbb{R}^4 , \mathbb{C}^{k+1}) $ functions.  The problem can be localized as follows. Suppose that $\{\rho_\nu\}$ is a unit partition subordinated to a finite covering $\{ \mathcal{{U}}_\nu\}$ of $\overline{\Omega}$ such that either $\overline{\mathcal{{U}}_\nu}\subset\Omega$  or $ {\mathcal{{U}}_\nu}\cap \partial \Omega\neq\emptyset$. Let $I'$ and $I''$ be the sets of corresponding indices $\nu$, respectively.  Write $f_\nu:=\rho_\nu f$.
 Then $\sum_{\nu\in I'\cup I''} f_\nu=f$, and
 \begin{equation}\label{eq:Lu=f}
    \mathcal{ L} f_\nu =u_\nu \qquad {\rm with }\qquad u_\nu =\rho_\nu \mathcal{ L} f+\mathcal{ L}\rho_\nu \cdot f.
 \end{equation}
   Then $ \widetilde{\rho_\nu f}=\rho_\nu \widetilde{ f}$ by definition, and so
$ \mathcal{ L}\widetilde{f_\nu}=\rho_\nu\mathcal{ L}\widetilde{f}+\mathcal{ L}\rho_\nu\widetilde{f }=\rho_\nu\widetilde{u} +\mathcal{ L}\rho_\nu\widetilde{f }=\widetilde{u}_\nu$. Hence \begin{equation}\label{eq:tilde-L}
\mathcal{ L}\widetilde{f_\nu}=\widetilde{ \mathcal{ L}f_\nu}.
  \end{equation}

For $\nu\in I' $,
 by Friderich's lemma, there exists a sequence $f_{\nu;n}\in C_0^\infty (\mathcal{{{U}}}_{\nu } , \mathbb{C}^{2k})$ such that $ f_{\nu;n}\rightarrow    f_\nu$,  $\mathcal{L}f_{\nu;n}\rightarrow  \mathcal{L}f_\nu$, $ \mathcal{M}f_{\nu;n}\rightarrow  \mathcal{M}f_\nu $ in $L^2( \mathcal{{{U}}}_{\nu } , \mathbb{C}^{2k})$.

For $\nu\in I''$, note that there exists a diffeomorphism ${\mathscr F}_\nu$ from   $\mathcal{{U}}_\nu$ to a neighborhood $ {{U}}_\nu$ of the origin in $\mathbb{R}^4$ such that
 the boundary $\mathcal{{U}}_\nu\cap\partial\Omega$ is mapped to the hyperplane $\{y_4=0\}$ and $\mathcal{{U}}_\nu\cap \Omega$ is mapped to ${{U}}_{\nu-}$. Here we denote by $y=(y_1,\ldots y_4)$ the  coordinates of $U_\nu\subset \mathbb{R}^4$ and $y={\mathscr F}_\nu(x)$. Let $L:=\mathscr F_{\nu*}\mathcal{ L}$, $M:=\mathscr F_{\nu*}\mathcal{ M}$ be differential operators by pushing forward.  Then by definition, we have
 \begin{equation*}
    L=\sum_{k=1}^4\sum_{j=1}^4\mathcal{ L}_jJ_{jk}(y)\frac \partial{\partial y_k}+\mathcal{ L}_0\left({\mathscr F}_\nu^{-1}(y)\right) ,\qquad  M=\sum_{k=1}^4\sum_{j=1}^4\mathcal{ M}_jJ_{jk}(y)\frac \partial{\partial y_k}
 \end{equation*}
 where $(J_{jk}(y))=(\frac{\partial y_k}{\partial x_j})$ is the Jacobian matrix. We claim that
 \begin{equation}\label{eq:claim}
   \dim \ker L_4(y)\equiv k-1\qquad   {\rm and} \qquad   \dim(\ker L_4(y)\cap\ker M_4(y))\equiv 0,
 \end{equation} for $y\in U_\nu$, where
  \begin{equation}\label{eq:L4-M4}
   L_4(y)=\sum_{j=1}^4\mathcal{ L}_jJ_{j4}(y) ,  \qquad  M_4=\sum_{j=1}^4\mathcal{ M}_jJ_{j4}(y).
 \end{equation}
Namely, our $L$ and $M$ satisfy the assumption of Proposition \ref{prop:minimal-domain2}.

 Now define  functions $h_\nu:=({\mathscr F}_\nu^{-1})^*f_\nu\in L^2( {{U}}_{\nu-} , \mathbb{C}^{2k})$. Note that by the property of pulling back of distributions, we have
  \begin{equation*}
     {L} g={\mathscr F}_{\nu*}\mathcal{ L} (g)= ({\mathscr F}_\nu^{-1})^* ( \mathcal{ L} ({\mathscr F}_\nu^*g))\end{equation*}
  for a distribution $g$ on ${{U}}_{\nu }$. Then by pulling (\ref{eq:tilde-L}) back by ${\mathscr F}_\nu^{-1}$,  we get
  $
   {L}\widetilde{h_\nu}=\widetilde{ { L}h_\nu}
$ on $\mathbb{R}^4 $,  and obviously $h_\nu$ is also in the maximal domain of $ { M}$.
Without loss of generality, we can assume   that ${\rm supp}\, {h_\nu}$ is sufficiently close to the origin as required by Proposition \ref{prop:minimal-domain2}.
So we can apply  Proposition \ref{prop:minimal-domain2} to $h_\nu$ to find a sequence $h_{\nu;n}\in C_0^\infty ({{U}}_{\nu } , \mathbb{C}^4)$ such that their restrictions  to ${{U}}_{\nu-}$, denoted by $\dot{h}_{\nu;n} :=h_{\nu;n}|_{{{U}}_{\nu-}}$, are in the minimal domain of $L$, i.e.
 \begin{equation}\label{eq:minimal-L-}
   \widetilde{L\dot{h}_{\nu;n} }=L\widetilde{\dot{h}_{\nu;n} },\quad {\rm and} \quad  \dot{h}_{\nu;n}  \rightarrow  {h}_\nu , \quad L\dot{h}_{\nu;n}\rightarrow L {h}_\nu , \quad M\dot{h}_{\nu;n}\rightarrow M {h}_\nu  \quad {\rm in }\quad  L^2( {{U}}_{\nu-}  , \mathbb{C}^{2k}) .
 \end{equation}
 Now   pulling back to $\mathcal{U}_{\nu }$ by ${\mathscr F}_\nu$, we get functions $f_{\nu;n}:={\mathscr F}_\nu^*(h_{\nu;n})\in C_0^\infty (\mathcal{{{U}}}_{\nu } , \mathbb{C}^{2k})$ satisfying
 \begin{equation*} \widetilde{\mathcal{ L} \dot{f}_{\nu;n} }=\mathcal{L}\widetilde{\dot{f}_{\nu;n}},\quad {\rm and} \quad
    \dot{f}_{\nu;n}\rightarrow  f_\nu , \quad \mathcal{L}\dot{f}_{\nu;n}\rightarrow \mathcal{L}f_\nu , \quad  \mathcal{M}\dot{f}_{\nu;n}\rightarrow \mathcal{M}f_\nu \quad {\rm in }\quad  L^2( \mathcal{{{U}}}_{\nu }\cap \Omega , \mathbb{C}^{2k}) ,
 \end{equation*}
     by pulling back (\ref{eq:minimal-L-}), where $\dot{f}_{\nu;n} :=f_{\nu;n}|_{\Omega  }={\mathscr F}_\nu^*(\dot{h}_{\nu;n})$ is the restriction of $f_{\nu;n}$ to $\Omega$. So by Proposition \ref{prop:minimal-domain}, $\dot{f}_{\nu;n} $ is in the  minimal domain of $\mathcal{ L}$. Then the finite sum $f_n=\sum_{\nu\in I'} f_{\nu;n} +\sum_{\nu\in I''} f_{\nu;n}\in C^\infty (\overline{\Omega }, \mathbb{C}^{2k})$ is also in the  minimal domain of $\mathcal{ L}$, and $ f_{ n}|_\Omega\rightarrow  f $  in $L^2( \Omega , \mathbb{C}^{2k})$.

  It remains to prove the claim (\ref{eq:claim}). For a fixed point $y$, write $\xi:=(J_{14}(y),\ldots,J_{44}(y) )\neq \mathbf{0}$. Comparing (\ref{eq:Lj-Mj}) with  (\ref{eq:L4-M4}), we see that $L_4$ and $M_4$ are exactly the matrices (\ref{eq:D0-k-nu-conj}) and (\ref{eq:mathcal-M-k}) with $\frac \partial{\partial x_j}$ replaced by $\xi_j$, respectively, i.e.
  \begin{equation}\label{eq:L-4}\begin{split}&
      L_4(y)={ \left(\begin{smallmatrix} \xi_3-\mathbf{i}\xi_4&
-\xi_1-\mathbf{i}\xi_2
 & 0  &0  &0  &0\hskip 6mm & \cdots\\ \frac { 1}k(\xi_1-\mathbf{i}\xi_2)& \frac { 1}k(\xi_3+\mathbf{i}\xi_4)& \frac {k- 1}k( \xi_3-\mathbf{i}\xi_4)&
\frac {k- 1}k(-\xi_1-\mathbf{i}\xi_2)
 &0  &0  & \cdots\\
0  &0  &\frac {2}k(\xi_1-\mathbf{i}\xi_2)& \frac {2}k(\xi_3+\mathbf{i}\xi_4)& \frac {k- 2}k( \xi_3-\mathbf{i}\xi_4)&
\frac {k- 2}k(-\xi_1-\mathbf{i}\xi_2)&  \cdots \\
0 &0  & 0  &0 &\frac {3}k(\xi_1-\mathbf{i}\xi_2)& \frac {3}k(\xi_3+\mathbf{i}\xi_4)&
 \cdots\\ \vdots   &\vdots   &\vdots  &  \vdots  & \vdots   & \vdots  &\vdots
 \end{smallmatrix}\right)_{(k+1)\times (2k)}  },\\& M_4(y)=\frac 12
    { \left(\begin{smallmatrix} - {\xi_1}+\textbf{i} {\xi_2}&-  {\xi_3} -\textbf{i} {\xi_4}& {\xi_3}-\textbf{i} {\xi_4}&- {\xi_1} -\textbf{i} {\xi_2} &0&0&0&\cdots\\
    0&0& - {\xi_1}+\textbf{i} {\xi_2}&-  {\xi_3} -\textbf{i} {\xi_4}& {\xi_3}-\textbf{i} {\xi_4}&- {\xi_1} -\textbf{i} {\xi_2}  &0&\cdots\\
    0&0&0&0& - {\xi_1}+\textbf{i} {\xi_2}&-  {\xi_3} -\textbf{i} {\xi_4}& {\xi_3}-\textbf{i} {\xi_4}&  \cdots\\ \vdots &\vdots &\vdots &\vdots &\vdots &\vdots&\vdots &\vdots
\end{smallmatrix}\right)_{(k-1)\times (2k)}  }  ,
\end{split} \end{equation}
  by (\ref{eq:k-CF-raised-2}) and (\ref{eq:k-CF-raised}), respectively. $ L_4(y)$
   is obviously   of rank $k+1$ for any $0\neq\xi\in \mathbb{R}^4$ by
  \begin{equation}\label{eq:minimal-L}
    \det \left(\begin{array}{rr} \xi_3-\mathbf{i}\xi_4&
-\xi_1-\mathbf{i}\xi_2 \\ \xi_1-\mathbf{i}\xi_2& \xi_3+\mathbf{i}\xi_4
\end{array}\right)=|\xi|^2  \qquad {\rm and } \qquad \left(\begin{array}{rr}  \xi_1-\mathbf{i}\xi_2& \xi_3+\mathbf{i}\xi_4
\end{array}\right)\quad     {\rm nondegenerate}.
  \end{equation}  and so   $\ker L_4(y)$ is of dimension  $k-1$ for any $\mathbf{0}\neq\xi\in \mathbb{R}^4$.
  For $  (a_1,\ldots,a_{2k})^t\in \ker L_4(y)\cap\ker M_4(y)$, it is easy to see
   that
   \begin{equation*}
      \left(\begin{array}{rr} \xi_3-\mathbf{i}\xi_4&
-\xi_1-\mathbf{i}\xi_2 \\ \xi_1-\mathbf{i}\xi_2& \xi_3+\mathbf{i}\xi_4
\end{array}\right) \left(\begin{array}{r} a_1 \\ a_2
\end{array}\right) =\mathbf{0}
   \end{equation*}
  by comparing the second row of $L_4(y)$ with the first row of $M_4(y)$. Therefore $(a_1,a_2)=(0,0)$, and by repeating this procedure, we get $(a_1,\ldots,a_{2k})=\mathbf{0}$.   The result $ \dim(\ker L_4(y)\cap\ker M_4(y))= 0 $ is proved.

\subsection{ Proof of Theorem \ref{thm:canonical}}

\begin{prop}\label{prop:self-adjoint} For $k=2,3,\ldots$, the associated Laplacian operator $\Box_\varphi$ in (\ref{eq:Box-varphi}) is a densely-defined, closed, self-adjoint and non-negative operator on $L_\varphi^2( \Omega,  \odot^{k-1}\mathbb{C}^{2 } \otimes \mathbb{C}^{2 }
)$.
\end{prop}
  The proof  is exactly the same as the proof of proposition 4.2.3 of \cite{CS} for $\overline{\partial}$-complex once we have  the following estimate (\ref{eq:box-estimate}).  See proposition 3.1 in \cite{Wa-weighted} for a complete proof for the  $k$-Cauchy-Fueter complexes  on  weighted $L^2$  space over $\mathbb{R}^{4n}$.
\vskip 2mm
  {\it Proof of Theorem \ref{thm:canonical}.} (1) The   $L^2$ estimate (\ref{eq:L2-k}) in Theorem \ref{thm:L2-estimate} implies that
      \begin{equation*}
          C_0 \left \| g \right\|^2_{  \varphi  }  \leq       \left\| {\mathscr D}_0^*g  \right\|^2_{\varphi } + \left\|{\mathscr D}_1g \right\|^2_{\varphi }=( \Box_{\varphi } g,g )_\varphi\leq\left \| \Box_{\varphi } g \right\|_{  \varphi  }\left  \| g \right\|_{  \varphi  } ,
      \end{equation*} for $  g\in {\rm Dom} (\Box_{\varphi })$,   i.e.
      \begin{equation}\label{eq:box-estimate}
   C_0 \left  \| g \right\|_{  \varphi  }\leq \left \| \Box_{\varphi } g\right\|_{  \varphi  }.
      \end{equation}
 Thus $\Box_{\varphi }$ is injective. This together with   self-adjointness of  $\Box_{\varphi }$    in Proposition \ref{prop:self-adjoint} implies the density of the  range (cf. section 2 of chapter 8 in \cite{RN} for this general property of a densely defined injective self-adjoint operator). For   fixed $f\in  L_\varphi^2( \Omega,  \odot^{k-1
}\mathbb{C}^{2 } \otimes \mathbb{C}^{2 }
)$,  the complex anti-linear functional
 \begin{equation*}
    \lambda_f:\Box_{\varphi } g \longrightarrow \langle f, g\rangle_\varphi\end{equation*}
  is then well-defined on the  dense subspace $\mathcal{R}(\Box_{\varphi })$ of $ L_\varphi^2( \Omega,  \odot^{k-1
}\mathbb{C}^{2 } \otimes \mathbb{C}^{2 }
)$, and  is finite since
 \begin{equation*}
 | \lambda_f(\Box_{\varphi } g )|=  |\langle f, g\rangle_\varphi|\leq \|f\|_\varphi\|g\|_\varphi\leq  \frac 1{C_0} \|f\|_\varphi\|\Box_{\varphi } g\|_\varphi
 \end{equation*}
 for any $  g\in {\rm Dom} (\Box_{\varphi })$, by (\ref{eq:box-estimate}). So $\lambda_f$ can be uniquely extended a continuous complex anti-linear functional on $ L_\varphi^2( \Omega,  \odot^{k-1
}\mathbb{C}^{2 } \otimes \mathbb{C}^{2 }  )$. By the Riesz representation theorem, there exists a unique element $h\in  L_\varphi^2( \Omega,  \odot^{k-1
}\mathbb{C}^{2 } \otimes \mathbb{C}^{2 }  )$ such that $\lambda_f(k)=\langle h, k\rangle_\varphi$ for any $k\in L_\varphi^2( \Omega,  \odot^{k-1
}\mathbb{C}^{2 } \otimes \mathbb{C}^{2 } )$,
 and $\|h\|_\varphi= |\lambda_f|\leq \frac 1{C_0}\|f\|_\varphi$. In particular, we have
 \begin{equation*}
 \langle h,\Box_{\varphi }g \rangle_\varphi =  \langle f, g\rangle_\varphi
 \end{equation*} for    any $g\in {\rm Dom} (\Box_{\varphi })$. This implies that $h \in {\rm Dom} (\Box_{\varphi }^*)$ and $\Box_{\varphi }^* h=f$. By self-adjointness of  $\Box_{\varphi }$ in Proposition \ref{prop:self-adjoint}, we find that  $h \in {\rm Dom} (\Box_{\varphi })$ and $\Box_{\varphi } h=f$. We write $h=N_\varphi f$. Then $\left \|N_\varphi f \right\|_{  \varphi  }\leq\frac 1{C_0} \left\| f \right\|_{  \varphi  }$.

      (2)  Since $N_\varphi f\in {\rm Dom} (\Box_{\varphi })$, we have ${\mathscr D}_0^*N_\varphi f\in {\rm Dom} ({\mathscr D}_0)$, ${\mathscr D}_1  N_\varphi f\in {\rm Dom} ({\mathscr D}_1^*)$, and
      \begin{equation}\label{eq:N-f}
          {\mathscr D}_0{\mathscr D}_0^*N_\varphi f=f-  {\mathscr D}_1^*{\mathscr D}_1  N_\varphi f
      \end{equation}by $\Box_{\varphi } N_\varphi f=f$.
      Because $f$ and ${\mathscr D}_0 u$ for any $u\in {\rm Dom} ({\mathscr D}_0)$  are both  ${\mathscr D}_1$-closed (${\mathscr D}_1\circ {\mathscr D}_0 =0$ by (\ref{eq:exact})), the above identity implies $ {\mathscr D}_1^*{\mathscr D}_1  N_\varphi f\in {\rm Dom} ({\mathscr D}_1)$   and  so
      $
{\mathscr D}_1 {\mathscr D}_1^*{\mathscr D}_1  N_\varphi u=0$  by   ${\mathscr D}_1$ acting on both sides of (\ref{eq:N-f}). Then
    \begin{equation*}
     0=\langle  {\mathscr D}_1 {\mathscr D}_1^*{\mathscr D}_1  N_\varphi f, {\mathscr D}_1  N_\varphi f\rangle_\varphi=\left\|    {\mathscr D}_1^*{\mathscr D}_1  N_\varphi f\right\|_\varphi^2,
    \end{equation*}
    i.e. $
       {\mathscr D}_1^*  {\mathscr D}_1 N_\varphi f=0.
   $ Hence ${\mathscr D}_0 {\mathscr D}_0^*N_\varphi f =f$ by (\ref{eq:N-f}). Moreover, we have ${\mathscr D}_0^*N_\varphi f\perp A^2_{(k)}(\mathbb{R}^{4n},\varphi)$ since $\langle v,{\mathscr D}_0^*N_\varphi f \rangle_\varphi=\langle{\mathscr D}_0 v , N_\varphi f \rangle_\varphi=0$ for any $ v\in A^2_{(k)}(\mathbb{R}^{4n},\varphi)$. The  estimate  (\ref{eq:canonical-est}) follows from
   \begin{equation*}\hskip 30mm
    \| {\mathscr D}_0^*N_\varphi f\|_\varphi^2+ \| {\mathscr D}_1 N_\varphi f\|_\varphi^2  =\langle \Box_{\varphi } N_\varphi f,N_\varphi f \rangle_\varphi\leq  \frac 1{C_0}\|  f\|_\varphi^2.\hskip 40mm \Box
   \end{equation*}

\end{document}